\documentclass[12pt]{article}

\usepackage[margin=1in]{geometry}
\usepackage{amsmath,amssymb,amsfonts,amsthm}
\usepackage{mathtools}
\usepackage{graphicx}
\usepackage{enumitem}
\usepackage{placeins}
\usepackage{gensymb}
\usepackage{csquotes}
\usepackage{subcaption}
\usepackage{booktabs}
\usepackage{tabularx}
\usepackage[thinc]{esdiff}
\usepackage[hidelinks]{hyperref}

\setlist[enumerate]{leftmargin=.5in}
\setlist[itemize]{leftmargin=.5in}

\numberwithin{equation}{section} 
\newtheorem{theorem}{Theorem}[section] 

\newtheorem{corollary}[theorem]{Corollary}
\newtheorem{proposition}[theorem]{Proposition}
\theoremstyle{definition}
\newtheorem{definition}[theorem]{Definition}

\newenvironment{relevance}
  {\begingroup\begin{abstract}}
  {\end{abstract}\endgroup}
\newenvironment{mathcontent}
  {\begingroup\begin{abstract}}
  {\end{abstract}\endgroup}
\newenvironment{keywords}
  {\begin{quote}\small\noindent\textbf{Keywords.}\ }
  {\end{quote}}
\newenvironment{MSCcodes}
  {\begin{quote}\small\noindent\textbf{MSC codes.}\ }
  {\end{quote}}

\DeclareMathOperator{\sgn}{sgn}
\providecommand{\cI}{\mathcal{I}}
\providecommand{\R}{\mathbf{R}}
\providecommand{\cG}{\mathcal{G}}
\providecommand{\cC}{\mathcal{C}}
\providecommand{\cS}{\mathcal{S}}
\providecommand{\cH}{\mathcal{H}}
\providecommand{\cHB}{\mathcal{H}}
\providecommand{\cPH}{\mathcal{S}}
\providecommand{\hb}{\text{H}}
\providecommand{\ph}{\text{S}}
\providecommand{\dcup}{\mathbin{\dot{\cup}}}
\providecommand{\lp}{\biggl{(}}
\providecommand{\rp}{\biggr{)}}
\providecommand{\RN}[1]{\textup{\uppercase\expandafter{\romannumeral #1\relax}}}
\newcommand{\ignore}[1]{}

\title{Period Homeostasis Near Hopf Bifurcation}
\author{Steve Manns \quad Janet Best \quad Martin Golubitsky\\[0.5em]
\normalsize Department of Mathematics, The Ohio State University.}
\date{}

\hypersetup{
  pdftitle={Period Homeostasis Near Hopf Bifurcation},
  pdfauthor={Steve Manns, Janet Best, and Martin Golubitsky}
}

\begin{document}
\maketitle

\begin{abstract}
Homeostasis is a biological phenomenon in which a function of the state of the system remains approximately constant as an input parameter varies. Recent mathematical work has developed a framework for studying homeostasis using methods from singularity theory. In this framework, the requirement that the function of the state remain approximately constant is replaced by the condition that the derivative of the function with respect to the input parameter vanishes at an isolated point. This zero-derivative condition is called infinitesimal homeostasis. Much of the existing theory concerns steady states. In this paper, we develop an analogous theory for oscillatory systems in which the homeostatic quantity is the period of a stable periodic solution. We refer to this behavior as period homeostasis and to the corresponding zero-derivative condition as infinitesimal period homeostasis. Experimental studies in cyanobacteria and cultured human cells suggest that circadian rhythms can disappear through Hopf bifurcation as temperature decreases to a critical value, with the oscillation amplitude tending to zero. Motivated by this evidence, we focus on stable periodic solutions arising from nondegenerate Hopf bifurcation. We show that bifurcation theory and singularity theory can be used to find period homeostasis in parameterized ODE models. Our results connect the onset of oscillations with homeostasis of the period and yield a geometric description of parameter space that locally organizes the qualitative behavior of solutions of the ODE.
\end{abstract}

\begin{relevance}
Certain biological oscillators must maintain stable timing despite changes in environmental or intracellular conditions. Given their role in coordinating important physiological processes, circadian clocks are a classic example. Their periods remain approximately 24 hours despite variations in ambient temperature, nutrient availability, and, in mammals, transcription rates and gene expression levels. This work develops a general mathematical framework for understanding how such period homeostasis can arise in ODE models of biological oscillators. The framework has several advantages: the molecular mechanism responsible for compensation need not be specified in advance; compensation is directly connected to the existence of oscillations; and the framework is parameter-agnostic in the sense that a variety of compensation phenomena, including those mentioned above, can be viewed as special cases of the theory. Simulations of a simple mathematical model of the mammalian circadian clock support the applicability of the framework by identifying parameter regimes in which period homeostasis occurs. 
\end{relevance}

\begin{mathcontent}
This work studies parameterized ODEs with a bifurcation parameter $\lambda$ and an input parameter $\cI$, focusing on stable periodic solutions born at a nondegenerate Hopf bifurcation. Liapunov--Schmidt reduction simplifies the problem of finding small-amplitude periodic orbits to that of finding zeros of an $S^1$-equivariant map on $\R^2$. The reduced equations uniquely determine the period of the oscillation as a smooth function of $\lambda$ and $\cI$, allowing derivatives of the period to be calculated from the reduced system. By varying $\lambda$ and $\cI$, the defining conditions for Hopf bifurcation and infinitesimal period homeostasis, respectively, can be satisfied. Under appropriate nondegeneracy conditions, the implicit function theorem yields smooth Hopf bifurcation and infinitesimal period homeostasis curves that intersect transversely. Their union forms a Hopf--Homeostasis graph that organizes the local dynamics by separating steady-state and periodic regimes and identifying parameter values at which the variation in the period with respect to $\cI$ is at its smallest. Singularity theory gives the normal form of the period input-output function near the infinitesimal period homeostasis curve. 
\end{mathcontent}

\begin{keywords}
period homeostasis, Hopf bifurcation, biological rhythms, circadian rhythms, singularity theory
\end{keywords}

\begin{MSCcodes}
34C23, 34C25, 37N25, 92B25
\end{MSCcodes}

\section{Introduction} \label{intro_sec}

This Introduction begins with a review of steady-state infinitesimal homeostasis and its generalization to period homeostasis, which is the principal subject of this paper. The Introduction continues with a description of Hopf bifurcation, its coexistence with period homeostasis,  and a way to find this coexistence.  The Introduction ends with a list of the remaining sections of the paper.

\subsection{Infinitesimal Steady-State Homeostasis}\label{intro_steady_state_homeo}

Claude Bernard's observations of regulation within the \emph{milieu int\'erieur} \cite{Bernard1865} provided an early foundation for the study of homeostasis. Walter Cannon \cite{Cannon1929, Cannon1939} later coined and developed the term \enquote{homeostasis} to describe the capacity of biological systems to maintain a steady internal condition despite disturbances caused by changes in the environment \cite{Cannon1939}. At the organismal scale, familiar examples include the regulation of body temperature, blood glucose, and ionic concentrations \cite{Morrison1946, Cannon1929}. In many such examples, the quantity being regulated is associated with a steady state.

This paper builds on recent work that uses mathematical models of biological phenomena to study homeostasis. For example, see the work of Nijhout, Reed, Best, and collaborators \cite{Nijhout2004, Best2009, Nijhout2014B, Nijhout2015, Nijhout2018} that consider biochemical networks of metabolic signaling pathways. Such models often take the form of parameterized ODEs. In this paper, we specifically consider
	\begin{align}
		\diff{u}{t} &= F(u, \lambda, \cI)
		\label{homeo_defn_eqn}
	\end{align}
where $u\in \R^n$ is the vector of \textit{state variables}, $\lambda\in\R$ is a \textit{bifurcation parameter}, and $\cI\in \R$ is an \textit{input parameter}. Usually, the system depends on additional parameters, but these are suppressed for now. In this section, we also suppress $\lambda$ and follow Golubitsky and Stewart \cite{Golubitsky2017} in formulating an infinitesimal notion of homeostasis. 

Suppose $(u_0, \cI_0)$ is a linearly stable steady-state solution of \eqref{homeo_defn_eqn}. It follows from the implicit function theorem that there exists a family of stable steady-state solutions $\tilde{u}(\cI)$, defined on a neighborhood of $\cI_0$, such that $\tilde{u}(\cI_0) = u_0$. A smooth function $\psi:\R^n \to \R$ is called an \textit{observable}. Associated to $\psi$ and $\tilde{u}$ is the \textit{input-output function} $\zeta(\cI) = \psi(\tilde{u}(\cI))$. The input-output function plays a central role in the study of homeostasis, because it allows for different notions of \enquote{approximately constant} to be precisely formalized.

\begin{definition}\rm \label{steady_state_homeo_defn}
	Let $\zeta(\cI)$ be an input-output function. We say that $\zeta(\cI)$ exhibits
		\emph{infinitesimal homeostasis} at the point $\cI_0$  if
			\begin{equation} 	\label{ss_infinitesimal_homeo}
				\zeta'(\cI_0) = 0
			\end{equation}
where $'$ is $\frac{\partial}{\partial \cI}$.
\end{definition}

Definition \ref{steady_state_homeo_defn} provides a local formulation of homeostasis in terms of critical points of the input-output function. If infinitesimal homeostasis occurs at $\cI_0$, then Taylor's theorem implies that homeostasis occurs on a neighborhood of $\cI_0$. Motivated by the observations of Nijhout et al.~\cite{Nijhout2014}, Golubitsky and Stewart \cite{Golubitsky2017} introduced this infinitesimal point of view and showed that it allows for methods from singularity theory to be applied. In particular, the normal form and universal unfolding of the input-output function can be determined. The infinitesimal approach is further supported by the analysis of Reed et al., which shows that these methods work in settings such as feedforward excitation \cite[Theorem 2.2]{Reed2017} and kinetic homeostasis \cite[Theorem 4]{Reed2017}.

\subsection{Infinitesimal Period Homeostasis}\label{intro_period_homeo}

Section \ref{intro_steady_state_homeo} offers a glimpse of infinitesimal homeostasis theory for steady-state solutions of \eqref{homeo_defn_eqn}. Further references for the steady-state theory may be found in \cite{Golubitsky2020, Golubitsky2018, Wang2021, Huang2022, Duncan2024}. Although the steady-state setting is supported by a growing body of results, analogous principles for periodic solutions of \eqref{homeo_defn_eqn} are still in the early stages of development. This direction is of particular interest because oscillatory phenomena occur ubiquitously in biology. Some examples from Forger~\cite{Forger2017} are listed in Table \ref{intro_table1}.

         \begin{table}[h!]
            \centering
            \begin{tabular}{|l|l|}
                \hline
                 \textbf{Biological Rhythm} & \textbf{Period}  \\
                 \hline
                 electrical pulses in neurons & millisecond\\
                 \hline
                 firefly flashing & second\\
                 \hline 
                 calcium waves in cells & minute\\
                 \hline
                 embryonic development & hour\\
                 \hline
                 tidal rhythms & hour\\
                 \hline
                 circadian clocks & day\\
                 \hline
                 menstrual  cycles & month\\
                 \hline
            \end{tabular}
            \caption{Examples of biological rhythms and the orders of their periods \cite[p. 16]{Forger2017}.}
            \label{intro_table1}
        \end{table}
        \FloatBarrier

	Notably, the periods of some of these rhythms are known to exhibit period homeostasis \cite{BellPedersen2005}, by which we mean that the period is approximately constant on variation of an input parameter. A classic example is the temperature-compensation property of circadian rhythms \cite{Dunlap2004, Forger2017, Hatakeyama2012}. Here, the period remains approximately 24 hours despite changes in ambient temperature. The period of circadian rhythms is also known to be homeostatic with respect to changes in trophic conditions, which is called nutrient compensation \cite{Phong2013}. Temperature and nutrient compensation are universally conserved across a wide range of organisms \cite{Hatakeyama2014}. For mammals specifically, the period of the circadian clock is robust with respect to changes in transcription rates \cite{Dibner2009} and gene expression levels \cite{Kim2012}. 

Experimental studies also suggest that circadian rhythms in cyanobacteria and cultured human cells disappear through Hopf bifurcation as temperature decreases to a critical value, with the oscillation amplitude decreasing toward zero \cite{Murayama2017, Xiao2026}. See Section \ref{applications_empirical_sec} for further discussion. Taken together, these observations motivate a theory connecting period homeostasis with the bifurcation governing the existence of the oscillation.

	Our formulation of period homeostasis follows Antoneli~et~al.~\cite{Antoneli2025}. 
	Suppose $u_0(t)$ is a stable periodic solution of \eqref{homeo_defn_eqn} at 
	($\lambda_0,  \cI_0$) with minimal period $P_0$. 
	By the continuation of hyperbolic periodic solutions theorem, there is a smooth family of periodic solutions 
	$u(\cdot, \lambda, \cI)$ defined on a neighborhood of $(\lambda_0, \cI_0)$ with smoothly varying  minimal period $P(\lambda,\cI)$. 		This family of periodic solutions satisfies $u(t, \lambda_0, \cI_0) = u_0(t)$ and 
	$P(\lambda_0, \cI_0) = P_0$. 

\begin{definition}\rm \label{period_io_defn}
Let $u(t, \lambda, \cI)$ be the periodic solution of \eqref{homeo_defn_eqn}  obtained by continuation from the  stable periodic solution $u_0(t)$. The associated map
\begin{align*}
\cI \longmapsto P(\lambda_0, \cI)
\end{align*}
where $P(\lambda_0, \cI)$ is the minimal period of $u(t, \lambda_0, \cI)$ is called the \emph{period input-output function}.
\end{definition}

The following definition formalizes the infinitesimal notion of period homeostasis studied in this paper.

\begin{definition} \rm \label{periodic_homeo_defn}
Let $P(\lambda_0, \cI)$ be a period input-output function. We say that $P(\lambda_0, \cI)$ exhibits \emph{infinitesimal period homeostasis} (IPH) at $\cI_0$ if
\begin{equation} 
P_{\cI}(\lambda_0, \cI_0) = 0
\label{per_infinitesimal_homeo}
\end{equation}
If, in addition to \eqref{per_infinitesimal_homeo},
\begin{align}
P_{\cI \cI}(\lambda_0, \cI_0)\neq 0
\end{align}
then we say that $P(\lambda_0, \cI)$ exhibits \emph{simple infinitesimal period homeostasis} (SIPH) at $\cI_0$.
\end{definition}

It follows from Definition~\ref{periodic_homeo_defn} that, here, the period of a periodic solution of \eqref{homeo_defn_eqn} is the function of interest in terms of homeostatic behavior. Other contexts have motivated considering different functions of a periodic solution of \eqref{homeo_defn_eqn}. For example, Yu and Thomas \cite{Yu2022} define a notion of infinitesimal homeostasis for periodic solutions in terms of the average value over the period. Using a generalization of the infinitesimal shape response curve (iSRC) introduced by Wang et al. in \cite{Wang2021B}, Yu and Thomas derive a formula for the derivative of their input-output function with respect to the input parameter $\cI$.  For periodic phenomena such as circadian rhythms where the period is homeostatic, the notion of IPH is more appropriate. 

Our definition of IPH follows \cite{Antoneli2025}, but we take a different approach. We focus on periodic solutions that are spawned from a nondegenerate Hopf bifurcation. The advantage of this approach is that it allows us to exploit the $S^1$ symmetry of Liapunov--Schmidt normal form. In this setting, singularity-theoretic methods can be applied to reduce the problem of finding circadian-like behavior to calculating certain derivatives of the normal form equation (see Section \ref{intro_example}).

The main result is Theorem \ref{simple_thm} in Section \ref{intro_simultaneous_hopf_iph}. Roughly speaking, this theorem says: If there are a bifurcation parameter and an input parameter in the ODEs, then defining conditions for Hopf bifurcation and IPH can be satisfied generically. In a neighborhood of the singularity, circadian-like behavior is expected. This is because the neighborhood contains curves on which Hopf bifurcation and IPH occur.  The main mathematical tool used to reach these conclusions is the proof of Hopf bifurcation via Liapunov--Schmidt reduction and the implicit function theorem. 

Our main result, Theorem \ref{simple_thm}, is described in Section \ref{intro_simultaneous_hopf_iph}. In Section~\ref{intro_example}, we translate this result into an analytical procedure for finding IPH near Hopf bifurcation. Section~\ref{applications_sec} then illustrates the theory using an example of biological interest. Specifically, we look at a simple model of the mammalian circadian clock proposed by Kim and Forger \cite{Kim2012}. In dimensionless form, this model has two free parameters, which satisfy the requirements of our theory. The analysis of Pei et al. \cite{Pei2024} shows that this model undergoes Hopf bifurcation, so what we add pertains to the occurrence of period homeostasis. Using numerical simulation, we find three IPH curves. See Figure \ref{KF_fig1}.

\subsection{Hopf Bifurcation via Liapunov--Schmidt Reduction}\label{intro_hopf}

As described in Section \ref{intro_period_homeo}, the starting point for our theory is nondegenerate Hopf bifurcation. We now elaborate on this aspect of our approach by summarizing two results from bifurcation theory that are needed: reducing to an $S^1$-equivariant map on $\R^2$ via the Liapunov--Schmidt procedure and the Hopf bifurcation theorem.

Liapunov--Schmidt reduction is based on restating \eqref{homeo_defn_eqn} as an operator on loop space and introducing a \textit{period parameter} $\tau \in \R$. Rescale time in \eqref{homeo_defn_eqn} as $s = (1 + \tau)t$ and define $v(s) = u(t)$. Then the operator is
\begin{align}
\Phi(v, \lambda, \cI, \tau) = (1 + \tau) \diff{v}{s} - F(v, \lambda, \cI)
\label{operator_eqn}
\end{align}
where $v(s)$ is a $2\pi$-periodic state,  $\Phi: \cC_{2\pi}^1 \times \R \times \R \times \R\to \cC_{2\pi}$, and $\cC_{2\pi}$ ($\cC_{2\pi}^1$) is the space of continuous (continuously differentiable) $2\pi$-periodic functions. 
Note that $\Phi(v(s), \lambda, \cI, \tau) = 0$ if and only if  $u(t)$  is a  $\frac{2\pi}{1 + \tau}$-periodic solution of \eqref{homeo_defn_eqn}.

Before proceeding, we make an assumption that implies the kernel of the linearization of \eqref{operator_eqn} is two-dimensional.
\begin{enumerate}[label =\textbf{(H\arabic*)}, align = left]
\item \label{H1} \textbf{Simple eigenvalue condition:} Assume \eqref{homeo_defn_eqn} has a steady-state solution $u\equiv 0$ for all $(\lambda, \cI)$. Let $(\mathrm{d}F)_{0, \lambda, \cI}$ be the Jacobian matrix evaluated along the steady-state solution $u\equiv 0$. $(\mathrm{d}F)_{0,0,0}$ has simple eigenvalues $\pm i$ and all other eigenvalues have strictly negative real part. 
\end{enumerate}

Two standard simplifying assumptions are built into Hypothesis~\ref{H1}. First, the steady state has been translated to $u\equiv 0$. Second, time has been rescaled so that the critical eigenvalues are $\pm i$. If the original critical eigenvalues are $\pm \omega_0 i$, with $\omega_0>0$, then the period is scaled by a factor of $\omega_0^{-1}$. Finally, in conjunction with \ref{H3} stated below, the assumption that all remaining eigenvalues have strictly negative real part implies the periodic solutions born at the Hopf bifurcation are stable. 

Hypothesis \ref{H1}, together with symmetry on loop space, leads via Liapunov--Schmidt reduction to the following reduced mapping \cite[p.~344]{Golubitsky1985}.

\begin{theorem}\label{reduced_eqn_thm}
Small-amplitude periodic solutions of \eqref{homeo_defn_eqn} are in one-to-one correspondence with zeros of the $S^1$-equivariant system
\begin{align}
\varphi(x, y, \lambda, \cI,\tau) = p(x^2 + y^2, \lambda, \cI,\tau) 
\begin{bmatrix}
x\\
y
\end{bmatrix}
+ q(x^2 + y^2, \lambda, \cI, \tau)
\begin{bmatrix}
-y\\
x
\end{bmatrix}
\label{reduced_eqn}
\end{align}
Moreover, $p$ and $q$ satisfy
\begin{align}
p(0) = 0 \qquad p_\tau(0) = 0 \qquad q(0) = 0 \qquad q_\tau(0) =-1
\label{p_q_eqns}
\end{align}
\end{theorem}

The significance of Theorem \ref{reduced_eqn_thm} is that it simplifies the problem of finding small-amplitude periodic solutions of \eqref{homeo_defn_eqn} to finding zeros of an $S^1$-equivariant map on $\R^2$. In order to conclude such zeros exist, we must strengthen our assumptions about \eqref{homeo_defn_eqn}. For easy reference, we collect all of our hypotheses for nondegenerate Hopf bifurcation into the following definition.  

\begin{definition}\rm \label{hopf_assumption_defn}
We say that \eqref{homeo_defn_eqn} satisfies the \emph{Hopf hypotheses} if \ref{H1} holds and the following conditions are satisfied.
\begin{enumerate}[label =\textbf{(H\arabic*)}, align = left, resume]
\item \label{H2} \textbf{Hopf crossing condition:} Let 
\[
\sigma(\lambda, \cI) \pm i \omega(\lambda, \cI)
\]
be simple eigenvalues of $(\mathrm{d}F)_{0, \lambda, \cI}$ such that $\sigma(0) = 0$ and $\omega(0) = 1$. The real part of these eigenvalues satisfies 
 \[
 \sigma_\lambda(0)\neq 0
 \]
\item \label{H3} \textbf{Cubic condition:} Let $p$ be as in Theorem \ref{reduced_eqn_thm}, and let $z = x^2$. Then $p$ satisfies 
\[
p_z(0)  < 0
\]
\end{enumerate}
\end{definition} 

These hypotheses yield the second main result we need from bifurcation theory, which is the nondegenerate Hopf bifurcation theorem. Using the implicit function theorem, solve the equations 
\[
p(z, \lambda, \cI, \tau) = 0 \qquad \text{and} \qquad q(z, \lambda, \cI, \tau) = 0
\]
locally for $z$ and $\tau$ as unique smooth functions of $(\lambda,\cI)$.  We denote the resulting functions by
\[
z = \tilde{Z}(\lambda, \cI) \qquad \text{and}\qquad \tau = T(\lambda, \cI)
\]
Thus, $\tilde{Z}(\lambda, \cI)$ and $T(\lambda, \cI)$ are the squared-amplitude and period-parameter functions associated with the corresponding periodic orbit, respectively. For our purposes, the crucial point of the nondegenerate Hopf bifurcation theorem is not only that there exists a unique periodic orbit, but also that its period is uniquely and smoothly determined by $T(\lambda, \cI)$.

\begin{theorem}\label{hopf_thm}
Assume the Hopf hypotheses (Definition \ref{hopf_assumption_defn}) are satisfied. Then the steady state $u \equiv 0$ undergoes nondegenerate Hopf bifurcation at the origin. For $(\lambda, \cI)$ on the side of the bifurcation where the steady state is unstable, there is a unique stable small-amplitude periodic orbit whose period is denoted by
\begin{align}
P(\lambda, \cI) = \frac{2\pi}{1 + T(\lambda,\cI)}
\label{period_eqn}
\end{align}
\end{theorem}

 A proof of Theorem \ref{hopf_thm} is given in Section \ref{liapunov_schmidt_sec}.

\subsection{Simultaneous Hopf Bifurcation and Period Homeostasis}\label{intro_simultaneous_hopf_iph}

In Sections \ref{intro_period_homeo} and \ref{intro_hopf} we discussed SIPH and nondegenerate Hopf bifurcation, respectively. We now bring these two concepts together to provide a mathematical explanation for biological phenomena where the period is tightly regulated, such as circadian rhythms. 

Our main result shows that, locally, the dynamics of \eqref{homeo_defn_eqn} are characterized by a certain type of graph. Before stating the result, we define the relevant geometric objects and clarify the hypotheses needed to arrive at the desired conclusion. 

The first is the \textit{Hopf curve} $\cH$, along which the steady-state solution of \eqref{homeo_defn_eqn} loses stability through Hopf bifurcation to a small-amplitude periodic orbit. Formally, 
\begin{align}
\cH = \{(\lambda, \cI): \sigma(\lambda, \cI) = 0\}
\label{hopf_curve_defn}
\end{align}
Under the hypotheses of Theorem \ref{simple_thm}, $\cH$ is locally the graph of a unique smooth function $\cI_{\hb}(\lambda)$.

The second is the \textit{IPH curve} $\cS$, along which the period of the periodic solution is constant with respect to $\cI$. Mathematically,
\begin{align}
\cS = \{(\lambda,\cI): \text{ a periodic solution exists and } P_{\cI}(\lambda,\cI) = 0\}
\label{iph_curve_defn}
\end{align}
Under the hypotheses of Theorem \ref{simple_thm}, $\cS$ is also locally the graph of a unique smooth function $\cI_{\ph}(\lambda)$. 

We emphasize that points in $\cS$ correspond to points of IPH in 
\eqref{homeo_defn_eqn}. Moreover, together, the curves $\cH,\cS$ 
form the central geometric object studied in this paper.
\begin{definition}\rm \label{hh_graph_defn}
A \emph{Hopf--Homeostasis (HH) graph} is a subset of the $\lambda \cI$-plane consisting of a Hopf bifurcation curve $\cH$ and an IPH curve $\cS$.
\end{definition}

The key point is that when nondegenerate Hopf bifurcation and SIPH occur simultaneously at a point in the $\lambda \cI$-plane, the local dynamics are organized by an HH graph. To obtain this picture, we supplement the Hopf hypotheses (Definition \ref{hopf_assumption_defn}) with conditions that ensure $\cH$ and $\cS$ are smooth curves that intersect transversely. As above, $\sigma$ denotes the real part of the critical eigenvalues, and $T$ is the unique smooth function that determines the period parameter $\tau$.

\begin{definition}\rm \label{siph_assumption_defn}
We say that \eqref{homeo_defn_eqn} satisfies the \emph{SIPH hypotheses} if each of the following conditions holds.
\begin{enumerate}[label = \textbf{(P\arabic*)}, align = left]
\item \label{P1}\textbf{Homeostasis crossing condition:} $\sigma_\cI(0) \neq 0$
\item \label{P2}\textbf{Curve condition:} $T_\cI(0) = 0 $ and $T_{\cI \cI}(0) \neq 0$
\item \label{P3} \textbf{Transversality condition:} $(\sigma_{\lambda}T_{\cI \cI} - \sigma_{\cI}T_{\cI \lambda})(0)\neq 0$
\end{enumerate}
\end{definition}

Conditions \ref{P1} and \ref{P2} ensure that, via the implicit function theorem, $\cH$ and $\cS$ are uniquely determined by the graphs of smooth functions of $\lambda$. Condition \ref{P3}  ensures that $\cH$ and $\cS$ meet transversely, which rules out pathological behavior such as the coincidence of the two curves. 

We now state the main result of this paper. It guarantees the existence of an HH graph, such as the one shown in Figure \ref{siph_example_subfigA}. A proof is given in Section \ref{siph_sec}. 

\begin{theorem}\label{simple_thm}
Suppose \eqref{homeo_defn_eqn} satisfies the Hopf (Definition \ref{hopf_assumption_defn}) and SIPH (Definition \ref{siph_assumption_defn}) hypotheses. Then in a neighborhood of the origin in the $\lambda \cI$-plane, \ref{S1}, \ref{S2}, and \ref{S3} hold:

\begin{enumerate}[label=\textbf{(S\arabic*)},align=left]
\item\label{S1} There is a unique smooth function $\cI_{\hb}(\lambda)$ such that $\cI_{\hb}(0) = 0$ and the Hopf curve is
\[
\cH = \{(\lambda,\cI): \cI = \cI_{\hb}(\lambda)\}
\]

\item\label{S2} There is a unique smooth function $\cI_{\ph}(\lambda)$ such that  $\cI_{\ph}(0) = 0$ and the IPH curve is
\[
\cS = \{(\lambda, \cI): \text{a periodic solution exists and } \cI = \cI_{\ph}(\lambda)\}
\]
Furthermore, $P_{\cI \cI}(\lambda, \cI_{\ph}(\lambda))\neq 0$. 

\item \label{S3}  $\cHB$ and $\cPH$ intersect transversely  at the origin.
\end{enumerate}

\end{theorem}

\begin{figure*}[h!]
            \centering
            \begin{subfigure}[b]{0.475\textwidth}
                \centering
            \includegraphics[width=\textwidth]{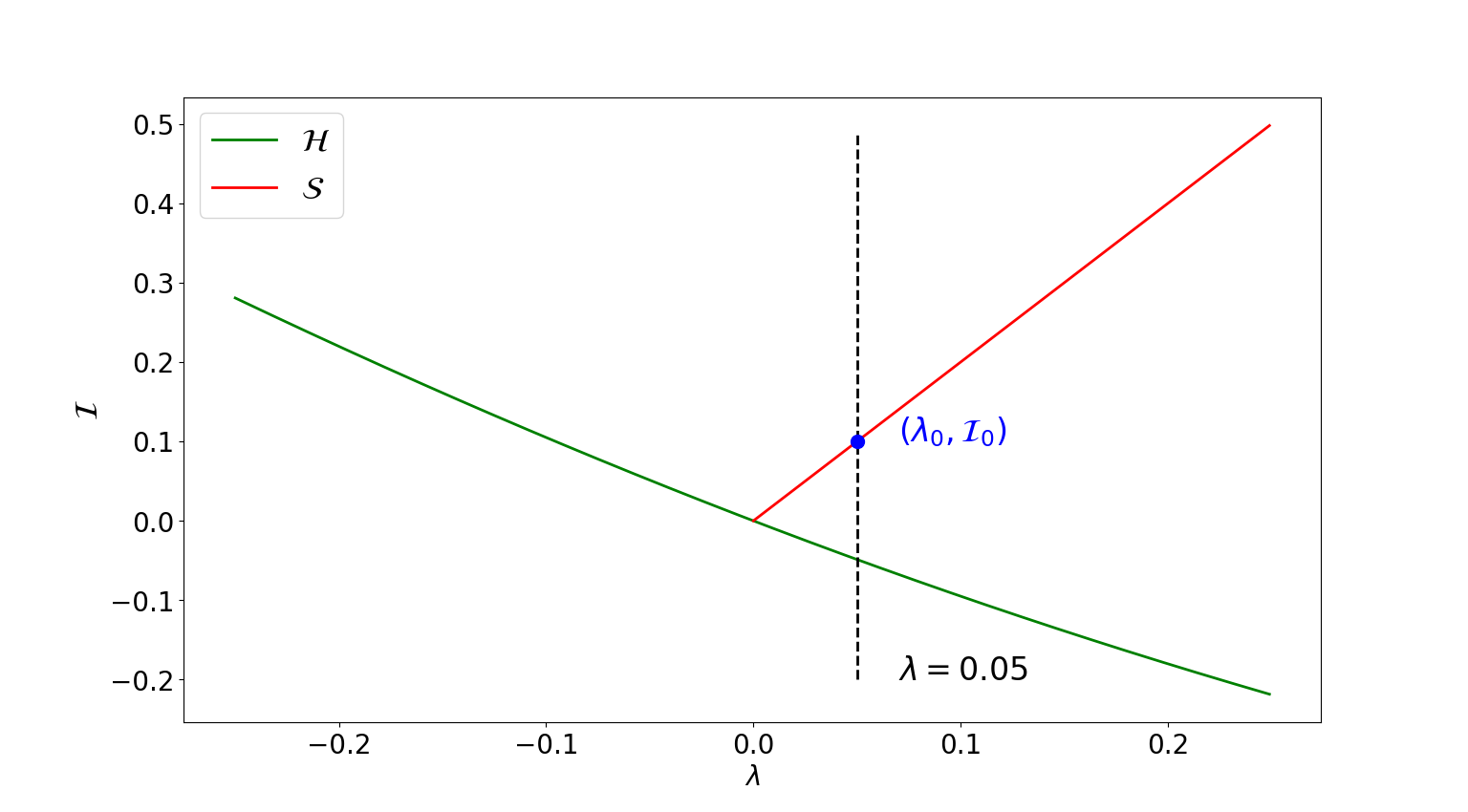}
	      \caption[]%
                {{\small} HH Graph}    
                \label{siph_example_subfigA}
            \end{subfigure}
            \hfill
            \begin{subfigure}[b]{0.475\textwidth}  
                \centering 
                \includegraphics[width=\textwidth]{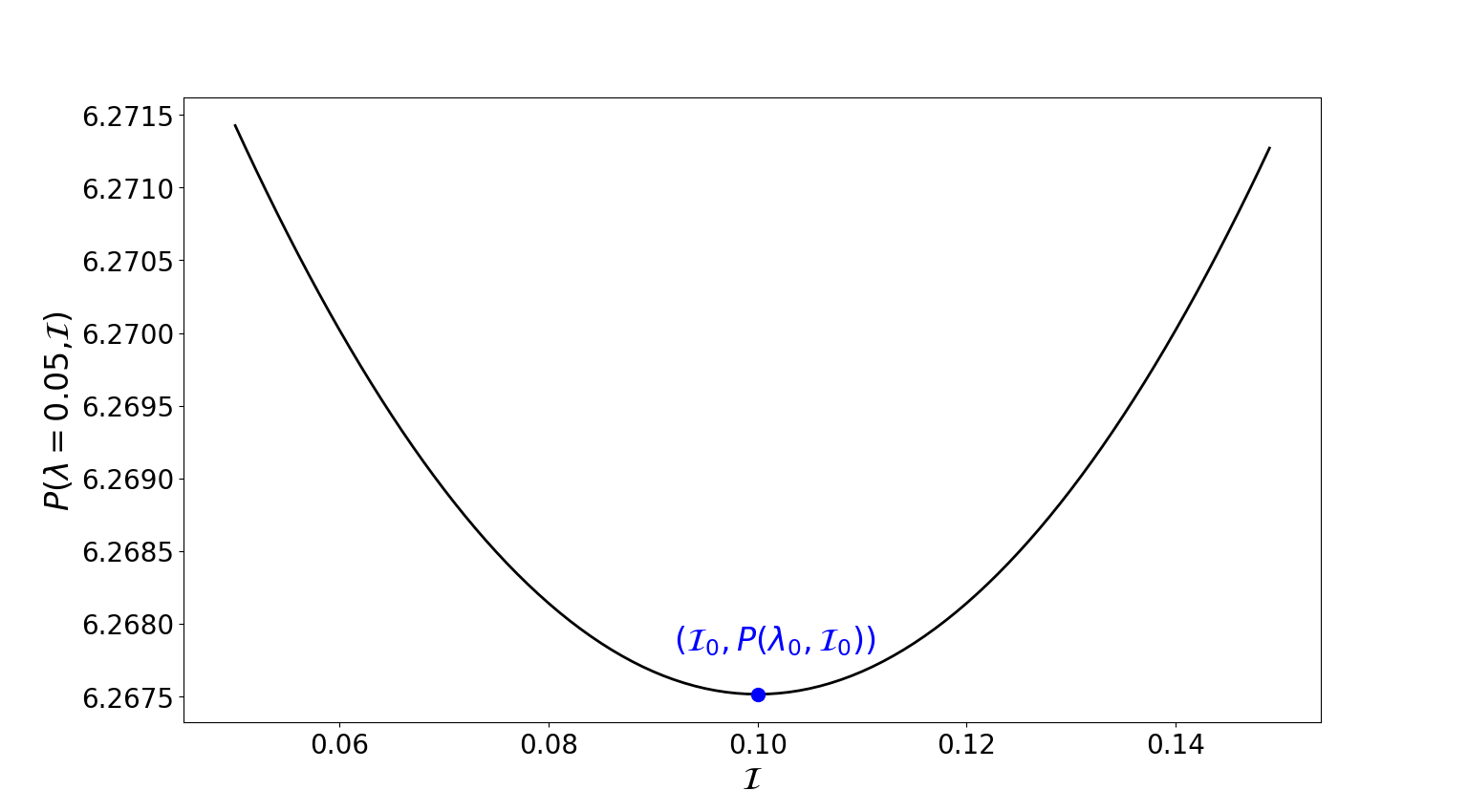}
	       \caption[]%
                {{\small }Period Input-Output Function}   
                \label{siph_example_subfigB}
            \end{subfigure}
            \caption{\small  A representative HH graph is shown in Figure~\ref{siph_example_subfigA}. Figure~\ref{siph_example_subfigB} is obtained by fixing $\lambda = 0.05$ in the parameter plane and plotting the corresponding period input-output function $P(0.05, \cI)$. The graph was produced from the Liapunov--Schmidt normal form \eqref{reduced_eqn} by choosing $p(r, \lambda, \cI, \tau) = \frac{1}{2}\cI^2 + \lambda \cI + \cI + \lambda - r$, $q(r, \lambda, \cI, \tau) = -\frac{3}{4}\cI^2 - \cI - \lambda - \tau + r$, and $r = x^2 + y^2$. The Hopf curve $\cH$ separates steady-state and periodic regimes, while the IPH curve $\cS$ identifies where the variation in the period with respect to $\cI$ is at its smallest. Thus, along a fixed-$\lambda$-slice, the period input-output function has smallest variation near the point where the slice meets $\cS$, illustrating SIPH. } 
\label{siph_example_fig}
\end{figure*}
\FloatBarrier

The key takeaway from Theorem \ref{simple_thm} is that when nondegenerate Hopf bifurcation and SIPH occur simultaneously, there generically exists an HH graph in which the Hopf curve $\cH$ and IPH curve $\cS$ intersect transversely. The Hopf curve $\cH$ separates the $\lambda \cI$-plane into steady-state and periodic regimes. Within the periodic regime, the IPH curve $\cS$ further divides the $\lambda \cI$-plane according to the monotonicity of the period with respect to $\cI$. Along $\cS$, the variation in the period is at its smallest. Together, these two curves provide a mathematical explanation for circadian-like behavior. Namely, the oscillation arises via nondegenerate Hopf bifurcation, while the approximately constant period is explained by SIPH. 

Figure \ref{siph_example_subfigA} illustrates the essential features described above that are true for any HH graph. In Section \ref{siph_sec}, we show that under the hypotheses of Theorem \ref{simple_thm}, every HH graph falls into one of four types, which are determined by the signs of $\sigma_{\cI}(0)$ and $T_{\cI \cI}(0)$. More precisely, Corollary \ref{sigma_i_cor} shows that the sign of $\sigma_{\cI}(0)$ determines which side of $\cH$ corresponds to the steady-state regime, while Corollary \ref{monotonicity_cor} shows that the sign of $T_{\cI \cI}(0)$ determines on which side of $\cS$ the period is increasing. As discussed in Section \ref{taylor_coeff_sec}, both quantities can be computed directly from $F$ in \eqref{homeo_defn_eqn}.

To state this classification concisely, let $\cH_{-}$ and $\cH_{+}$ denote the regions below and above the Hopf curve $\cH$, respectively. Similarly, let $\cS_{-}$ and $\cS_{+}$ denote the regions below and above the IPH curve $\cS$, respectively. 

We formalize the previous remarks in the next theorem, which also specifies the normal form of the period input-output function. This theorem follows directly from Corollaries \ref{sigma_i_cor}-\ref{normal_form_cor}. The latter are proved in Section \ref{siph_sec}.

\begin{theorem}\label{classification_thm}
Suppose \eqref{homeo_defn_eqn} satisfies the Hopf (Definition \ref{hopf_assumption_defn}) and SIPH (Definition \ref{siph_assumption_defn}) hypotheses. Then the HH graph and normal form of the period input-output function are classified by Table \ref{siph_table}.
\end{theorem}

\begin{table}[h!]
\centering

\renewcommand{\arraystretch}{1.25}
\setlength{\extrarowheight}{2pt}

\begin{tabular}{c c p{0.68\linewidth}}
\hline
$\operatorname{sgn}\sigma_{\mathcal I}(0)$ &
$\operatorname{sgn}T_{\mathcal I\mathcal I}(0)$ &
Qualitative consequences \\ 
\hline
$+$ & $+$ &
\textbf{Steady-State regime:} $\cH_{-}$ \\
& & \textbf{Periodic regime:} $\cH_{+}$\\
& & \textbf{Monotonicity:} period increases in $\cS_{-}$, decreases in $\cS_{+}$ \\
& & \textbf{Normal form:} $P(\mathcal I) = -\mathcal I^2$ \\ 
\hline
$+$ & $-$ &
\textbf{Steady-State regime:} $\cH_{-}$ \\
& & \textbf{Periodic regime:} $\cH_{+}$ \\
& & \textbf{Monotonicity:} period decreases in $\cS_{-}$, increases in $\cS_{+}$ \\
& & \textbf{Normal form:} $P(\mathcal I) = +\mathcal I^2$ \\
\hline
$-$ & $+$ &
\textbf{Steady-State regime:} $\cH_{+}$ \\
& & \textbf{Periodic regime:} $\cH_{-}$ \\
& & \textbf{Monotonicity:} period increases in $\cS_{-}$, decreases in $\cS_{+}$ \\
& & \textbf{Normal form:} $P(\mathcal I) = -\mathcal I^2$ \\
\hline
$-$ & $-$ &
\textbf{Steady-State regime:} $\cH_{+}$ \\
& & \textbf{Periodic regime:} $\cH_{-}$ \\
& & \textbf{Monotonicity:} period decreases in $\cS_{-}$, increases in $\cS_{+}$ \\
& & \textbf{Normal form:} $P(\mathcal I) = +\mathcal I^2$ \\
\hline
\end{tabular}
\caption{Classification of the HH graph and normal form of the period input-output function. The sign of $\sigma_{\cI}(0)$ determines whether the steady-state regime lies above or below the Hopf curve $\cH$, while the sign of $T_{\cI \cI}(0)$ determines whether the normal form of the period input-output function is $\cI^2$ or $-\cI^2$.}
\label{siph_table}
\end{table}
\FloatBarrier

\subsection{An Analytical Procedure for Finding Period Homeostasis Near Hopf Bifurcation}\label{intro_example}

Having described our results, we now summarize an analytical procedure for finding IPH. Starting with \eqref{homeo_defn_eqn}, the goal is to verify the hypotheses, then approximate the Hopf curve $\cH$ and IPH curve $\cS$ near the origin. The steps of the procedure are given in Table \ref{iph_procedure_table}. References are also provided for computing relevant derivatives or justifying certain steps. 

\begin{table}[htbp]
\centering
\small
{
\renewcommand{\arraystretch}{1.5}
\begin{tabularx}{\textwidth}{>{\raggedright\arraybackslash}p{0.07\textwidth}
                            >{\raggedright\arraybackslash}p{0.23\textwidth}
                            >{\raggedright\arraybackslash}X
                            >{\raggedright\arraybackslash}p{0.17\textwidth}}
\toprule
\textbf{Step} & \textbf{Task} & \textbf{Condition or calculation} & \textbf{Reference} \\
\midrule

1 &
Translate the steady state&Satisfy the first assumption of \ref{H1}: translate the steady state to $u\equiv 0$& N/A
\\

2 &
Locate Hopf bifurcation &
Check \ref{H1}: $\sigma(0) = 0$; after rescaling time, $\omega(0) = 1$
& N/A
\\

3 &
Reduce to normal form &
Under \ref{H1}, apply Liapunov--Schmidt reduction to obtain an $S^1$-equivariant map with coefficients $p$ and $q$
& Theorem \ref{reduced_eqn_thm}
\\

4 &
Verify Hopf crossing &
Check \ref{H2}: $\sigma_{\lambda}(0) \neq 0$
& Proposition 3.3 in \cite[p. 352]{Golubitsky1985}
\\

5 &
Verify the cubic condition &
Check \ref{H3}: $p_z(0) <  0$
& Proposition 3.3 in \cite[p. 352]{Golubitsky1985}
\\

6 &
Verify homeostasis crossing &
Check \ref{P1}: $\sigma_{\mathcal I}(0)\neq 0$
& Proposition 3.3 in \cite[p. 352]{Golubitsky1985}, with $\lambda$ replaced by $\cI$
\\

7 &
Locate IPH &
Check the defining condition in \ref{P2}: $T_{\cI}(0)=0$
& Proposition \ref{linear_t_prop}
\\

8 &
Verify nondegenerate IPH &
Check the nondegeneracy condition in \ref{P2}: $T_{\cI \cI}(0)\neq 0$
& Proposition \ref{quadratic_t_prop}
\\

9 &
Verify transversality &
Check \ref{P3}:
$(\sigma_{\lambda} T_{\cI \cI} - \sigma_{\cI} T_{\cI \lambda})(0)\neq 0$ 
& For $T_{\cI \lambda}(0)$, refer to Proposition \ref{quadratic_t_prop}
\\

10 &
Approximate the HH graph &
Use \ref{S1} and \ref{S2}:
$\cI_{\hb}(\lambda) \approx -\frac{\sigma_{\lambda}}{\sigma_{\cI}}(0)\lambda$ and
$\cI_{\ph}(\lambda)\approx - \frac{T_{\cI \lambda}}{T_{\cI \cI}}(0) \lambda$.
Condition \ref{S3} gives the transverse intersection
& Proof of Theorem \ref{simple_thm}
\\

\bottomrule
\end{tabularx}
}
\caption{Analytical procedure for detecting IPH near Hopf bifurcation.}
\label{iph_procedure_table}
\end{table}
\FloatBarrier

The formulas referenced in Table~\ref{iph_procedure_table} assume the simplifications in Step 1: the steady state has been translated to the origin, and time has been rescaled so that the critical eigenvalues are $\pm i$. To work in the original coordinates and time scale, one must express the derivatives of the simplified vector field in terms of derivatives of the original one. Such calculations are standard in the degenerate Hopf bifurcation literature. For example, see Farr et al.~\cite{Farr1989}.

\subsection{Outline of Paper}\label{intro_outline}

We end Section \ref{intro_sec} by outlining the remainder of the paper. Section \ref{applications_sec} reviews evidence for period homeostasis, including temperature compensation (Section \ref{applications_empirical_sec}) and simulations of a mammalian circadian clock model proposed by Kim and Forger \cite{Kim2012} (Section \ref{applications_model_sec}), which display the behaviors predicted by our theory. Section \ref{results_sec} develops the analysis by deriving a smooth period input-output function via Liapunov--Schmidt reduction (Section \ref{liapunov_schmidt_sec}), proving the main result (Section \ref{siph_sec}), and calculating relevant Taylor coefficients of $T$ (Section \ref{taylor_coeff_sec}) to help verify the SIPH hypotheses (Definition \ref{siph_assumption_defn}). Section \ref{conclusion_sec} concludes with a discussion of implications and directions for future work. 

\section{Support for Period Homeostasis} \label{applications_sec}
	This section answers a central question: \enquote{Why is the theory presented in this paper important?} A large portion of biological rhythms research focuses on circadian rhythms, which the reader will recall from Section \ref{intro_period_homeo} is a prime example of period homeostasis.  Section \ref{applications_empirical_sec} elaborates on the temperature-compensation property of circadian rhythms, with emphasis placed on empirical evidence. Section \ref{applications_model_sec} takes a more mathematical approach by studying a particular model of the mammalian circadian clock. The mathematical model discussed in Section \ref{applications_model_sec} was proposed by Kim and Forger \cite{Kim2012}, who point out that it is important for the period to be robust with respect to variation in gene dosage. Numerical simulations indicate that IPH helps explain this robustness.

    \subsection{Empirical Evidence: Temperature Compensation}\label{applications_empirical_sec}

	Before presenting empirical evidence in support of temperature compensation, let us clarify how the term \enquote{circadian rhythm} is used in this paper. Roughly speaking, circadian rhythms are those for which the process repeats approximately every 24 hours. Hence the name \enquote{circadian}, which comes from the Latin \enquote{circa}, meaning \enquote{around}, and \enquote{dies}, meaning \enquote{day}. However, we prefer the more formal point of view taken by Dunlap et al. \cite[p. 68]{Dunlap2004}, who say: \enquote{Circadian rhythms are defined by major, observable, and well-established criteria, not by a molecular mechanism.}

	A biological rhythm must satisfy three criteria in order to be called \textit{circadian}. First, the rhythm has a free-running period of approximately 24 hours. In other words, under constant conditions such as lighting and temperature, the period of the rhythm remains about 24 hours. Free-running values for most species are between 23 and 25 hours. Second, the rhythm is temperature-compensated, meaning that the free-running period length remains approximately constant over a range of ambient temperatures. Third, the rhythm is entrainable. This means that the rhythm is capable of synchronizing with external stimuli that have a similar period. An important stimulus for circadian rhythms is the 24-hour light-dark cycle resulting from Earth's rotation. To say that a circadian rhythm has been entrained to the light-dark cycle means that the two rhythms have the same period and a stable phase relationship exists. 

	Although these three characteristics make the notion of a circadian rhythm more precise, there is a peculiarity in this description. Namely, as the ambient temperature varies, chemical reactions proceed at different rates. For example, it is standard to try to extend the shelf life of food by storing it in a freezer. This observation makes it hard to imagine that a circadian clock can be temperature-compensated. Nonetheless, as proclaimed by chronobiologists, if this were not the case, then the circadian clock would be good only as a thermometer, not as a timekeeper. 
    
To make matters even more complex, it is incorrect to say that the ambient temperature does not affect circadian clocks. Indeed, one of the most important functions of a circadian clock is to provide an internal estimate of the external local time. One way in which this is accomplished is by identifying the periodic temperature variations associated with dawn and dusk. Thus, the period of a circadian clock is insensitive to temperature fluctuations, but the phase of the clock adapts to temperature variations with a near-circadian period. In fact, even below the critical temperature, near-circadian ambient temperature variations can rescue a circadian rhythm via resonance. With the additional experimental evidence that circadian rhythms continuously decrease in amplitude as a constant ambient temperature approaches the critical minimum, researchers have suggested not only that both cyanobacteria \cite{Murayama2017} and cultured human cells \cite{Xiao2026} exhibit Hopf bifurcations, but that Hopf bifurcation may be a universal feature of circadian rhythms.
    
	A metric that is frequently used for measuring the temperature dependence of biochemical reactions is the $Q_{10}$ quotient. The $Q_{10}$ quotient is defined as
            \begin{align*}
                Q_{10} = \frac{\text{reaction rate at temperature } (T+10)\degree \text{C}}{\text{reaction rate at temperature }T\degree \text{C}}
            \end{align*}
Common values for the $Q_{10}$ quotient are between two and three. However, Table \ref{I_table2} indicates this is not true for the free-runs of circadian rhythms, as the values reported are all close to one. Notice from the definition of the $Q_{10}$ quotient that values close to one indicate that the reaction is relatively insensitive to temperature fluctuations. Therefore, Table \ref{I_table2} supports the assertion that temperature compensation is a characteristic property of circadian rhythms. 
            \begin{table}[h!]
                \centering
                \begin{tabular}{|l|l|l|}
                    \hline
                     \textbf{Species}& \textbf{Phenomenon measured}  & $\mathbf{Q_{10}}$\\
                     \hline
                     \textit{Euglena} (protist) & Phototaxis  &1.01-1.1\\
                     \hline
                     \textit{Gonyaulax} (dinoflagellate) & Bioluminescence  & 0.85\\
                     \hline
                     \textit{Neurospora} (fungus) & Conidiation & 1.03\\
                     \hline
                     \textit{Drosophila} (insect) & Eclosion  &1.1-1.25\\
                     \hline
                     \textit{Lacerta} (lizard)& Locomotor activity & 1.02\\
                     \hline
                     \textit{Myotis} (bat) & Locomotor activity & 1.4\\
                     \hline
                     \textit{Peromyscus} (mouse) & Locomotor activity & 1.1-1.4 \\
                     \hline
                \end{tabular}
                \caption{Values of the $Q_{10}$ quotient for free-runs of circadian rhythms of representative species. Taken from Dunlap et al. \cite[p. 70]{Dunlap2004}. The original source is Sweeney and Hastings \cite{Sweeney1960}.}
                \label{I_table2}
            \end{table}

	Having reviewed the biological evidence for temperature compensation, we conclude by showing how our theory accounts for it, thereby addressing the central question of Section \ref{applications_sec}. Circadian rhythms are usually modeled as limit-cycle oscillators. For example, Ruoff and Rensing \cite{Ruoff1996} use a modified Goodwin oscillator \cite{Goodwin1966} to model temperature-compensated circadian clocks. In our framework, a stable limit cycle arises via the occurrence of nondegenerate Hopf bifurcation. Temperature compensation corresponds to period homeostasis in the terminology of this paper. The theory guarantees a curve of IPH points, which implies period homeostasis on a neighborhood of the curve. Therefore, the framework accounts for both essential features of temperature-compensated circadian clocks: the existence of a stable limit cycle and the robustness of the period to changes in ambient temperature. 

    \subsection{The Kim--Forger SNF Model}\label{applications_model_sec}
        In 2012, Kim and Forger \cite{Kim2012} described a modified version of a mathematical model of the mammalian circadian clock that was proposed by Forger and Peskin \cite{Forger2003}. Since the model reproduces a variety of experimental data, the Kim--Forger model has become prominent in circadian rhythm research. However, the Kim--Forger model is quite large, as it consists of 181 variables and 75 parameters. Consequently, in the same paper, Kim and Forger introduced three simpler mathematical models of the mammalian circadian clock for the purpose of analysis. This subsection shows how one of the three simple models fits into our framework. 

        The model we consider has a single-negative feedback loop (SNF) structure as described in \cite{Kim2012}. The ODEs are
            \begin{align}
                \begin{split}
                    \diff{M}{t} &= \alpha_1 \frac{A_{\text{free}}}{A_T} - \beta_1 M\\
                    \diff{P_c}{t} &= \alpha_2 M - \beta_2 P_c\\
                    \diff{P}{t} &= \alpha_3 P_c - \beta_3P\\
                    A_{\text{free}} &= \frac{1}{2}\lp A_T - P -K_d + \sqrt{(A_T-P-K_d)^2+4K_dA_T}\rp
                \end{split}
                \label{KF_eqn1}
            \end{align}
        In \eqref{KF_eqn1}, the three state variables are the concentrations of PER mRNA ($M$), PER protein in the cytoplasm ($P_c$), and PER protein in the nucleus ($P$). The BMAL1:CLOCK transcription factor is denoted by $A$, $A_T$ is the total concentration of BMAL1:CLOCK in the nucleus, and $A_{\text{free}}$ is the concentration of BMAL1:CLOCK in the nucleus that is not bound to PER:CRY. For a more detailed description of the Kim--Forger SNF model, the reader is referred to \cite{Kim2012}.

	The first step in demonstrating that the Kim--Forger SNF model fits into our framework is showing that a stable limit cycle arises from a nondegenerate Hopf bifurcation. Recently, Yao et al. \cite{Yao2022} and Pei et al. \cite{Pei2024} independently analyzed the bifurcations of the Kim--Forger SNF model. The results in \cite{Pei2024} are more comprehensive and serve as the starting point for our analysis. To increase the odds of an oscillation, it is standard to assume that $\beta_1 = \beta_2 = \beta_3$ in \eqref{KF_eqn1} \cite{Forger2011}. After making this assumption, \eqref{KF_eqn1} can be cast into dimensionless form as
            \begin{align}
                \begin{split}
                    \diff{M}{t} &= \frac{1}{2}\lp 1 - \frac{P}{A} - \frac{K_d}{A} + \sqrt{(1-\frac{P}{A} - \frac{K_d}{A})^2 + \frac{4K_d}{A}}\rp - M\\
                \diff{P_c}{t} &= M-P_c\\
                \diff{P}{t} &= P_c - P
                \end{split}
                \label{KF_eqn2}
            \end{align}
where $A, K_d>0$. Note that \eqref{KF_eqn2} has two free parameters.

We use the same bifurcation parameter as in \cite{Pei2024}, namely $K_d$, and take $A$ to be the input parameter. In light of Theorem \ref{simple_thm}, we search for a Hopf bifurcation curve $\cH$ and an IPH curve $\cS$. The result from \cite{Pei2024} that is crucial to our analysis is the identification of the Hopf bifurcation curve $\cH$. According to Theorem 3.1 in \cite{Pei2024}, \eqref{KF_eqn2} undergoes Hopf bifurcation on the curve
            \begin{align*}
                K_d = \frac{128A^2 + 240A + 49-(7+16A)\sqrt{256A + 49}}{16(1-8A)}
            \end{align*}
The first Liapunov coefficient is negative for $13/1616 < A < 1/8$ and positive for $A <13/1616$. The sign of the first Liapunov coefficient changes at the point $\mathcal{G}$, where
            \begin{align*}
                \mathcal{G}= (K_d, A)= \lp \frac{38491}{11312} - \frac{\sqrt{267338880}\sqrt{1030301}}{15423912}, \frac{13}{1616}\rp 
            \end{align*}
Theorem 3.2 in \cite{Pei2024} states that a generalized Hopf bifurcation occurs at $\mathcal{G}$. The interested reader is referred to Figure 3 in \cite{Pei2024} to see the different qualitative behaviors that \eqref{KF_eqn2} exhibits near this point. 

	The second step in demonstrating that the Kim--Forger SNF model fits into our framework is showing that the robustness of the period is explained by IPH. We proceed numerically. For a fixed value of the bifurcation parameter $K_d$, we vary the input parameter $A$ over an interval on which \eqref{KF_eqn2} admits a periodic solution. For each parameter pair $(K_d, A)$ in this grid, we numerically integrate \eqref{KF_eqn2}, compute the period, and plot the period as a function of $A$. Whenever IPH occurs at some $A$ (for the chosen $K_d$), we record the corresponding $(K_d, A)$. The outcome of this procedure is shown in Figure \ref{KF_fig1}.

            \begin{figure}[h!]
                \centering
                \includegraphics[scale = 0.4]{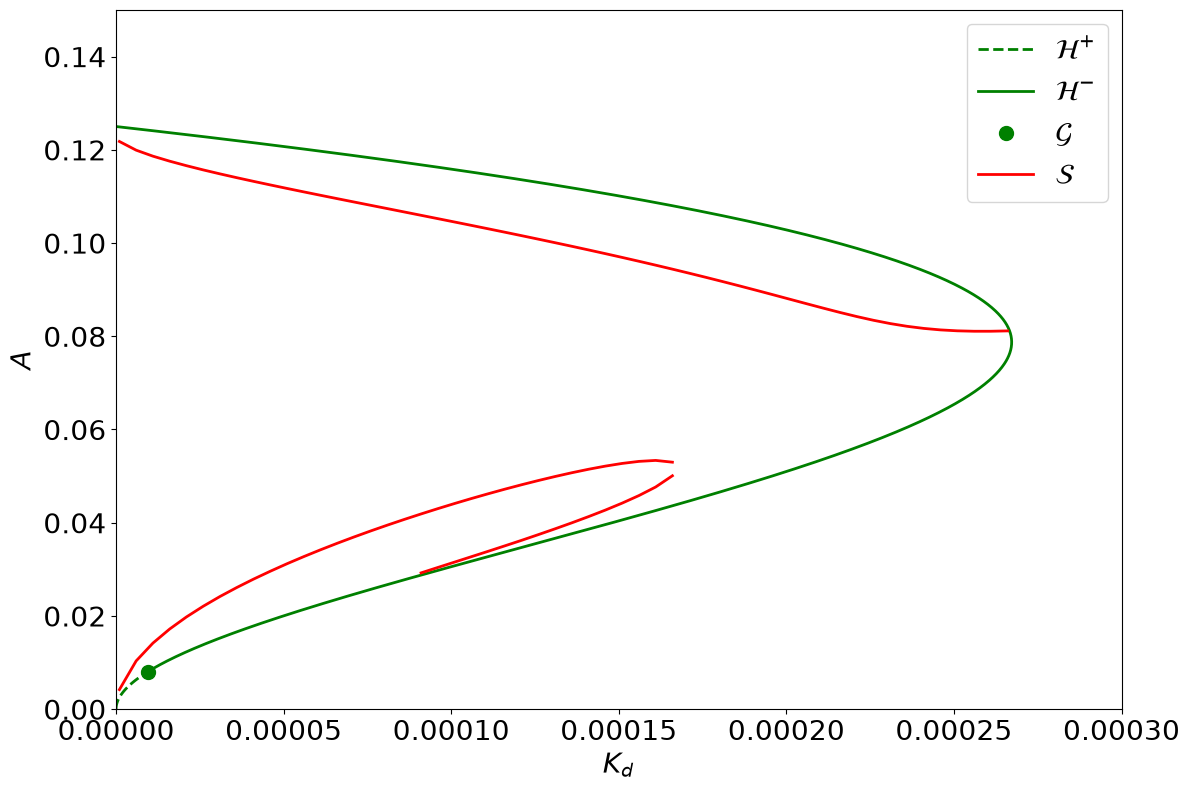}
                \caption{The parameter plane for the dimensionless Kim--Forger SNF model \eqref{KF_eqn2}. The Hopf curve $\cHB$ is shown in green. The first Liapunov coefficient is negative (positive) on the solid (dashed) part of this curve. A generalized Hopf bifurcation occurs where the sign of the first Liapunov coefficient changes. Each red curve $\cPH$ is an SIPH curve.}
                \label{KF_fig1}
            \end{figure}

	The important takeaway from Figure \ref{KF_fig1} is that the Kim--Forger SNF model exhibits IPH. More precisely, the three red curves in Figure \ref{KF_fig1} represent points in the $K_dA$-parameter plane where SIPH occurs. Since IPH implies period homeostasis, this finding supports the assertion that period robustness in the Kim--Forger SNF model is a consequence of IPH. Accordingly, the Kim--Forger SNF model, a well-renowned model of the mammalian circadian clock, fits into our framework and thus provides a concrete answer to the central question of Section \ref{applications_sec}.

	Beyond the aims of Section \ref{applications_sec}, our numerical results also reveal interesting global behavior that is not described by our local theory.  In the local setup, the Hopf and IPH curves are unique. To see this illustrated by the Kim--Forger SNF model, consider the HH graph shown in Figure \ref{kf_hh_graph_fig}. Globally, on the other hand, multiple IPH curves may exist. Multiple IPH curves lead to the possibility of intersections. Figure \ref{KF_fig1} indicates this happens in the Kim--Forger SNF model, as two SIPH curves approach one another as $K_d$ increases to $1.65\cdot 10^{-4}$, then both disappear. 

\begin{figure}[h!]
\centering
\includegraphics[scale = 0.4]{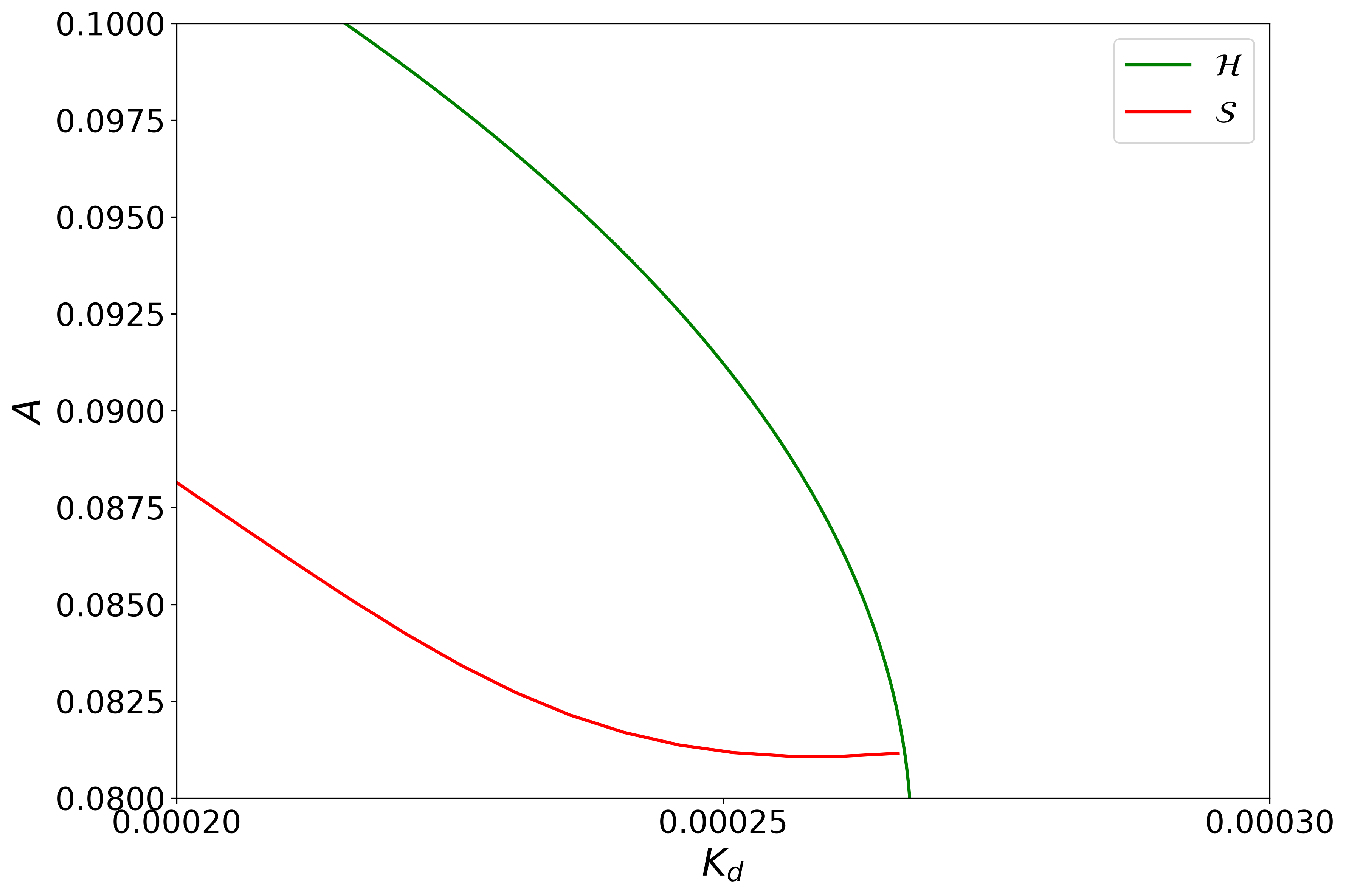}
\caption{\small An HH graph for the Kim--Forger SNF model, obtained by zooming in on an intersection of the Hopf curve $\cH$ and an IPH curve $\cS$ in Figure \ref{KF_fig1}. Here, $\cH_{-}$ is the periodic regime, and the period input-output function decreases in $\cS_{-}$ and increases in $\cS_{+}$. Thus, this graph corresponds to the case $\sigma_{\cI}(0) < 0$ and $T_{\cI \cI}(0) < 0$ in Theorem \ref{classification_thm}.}
\label{kf_hh_graph_fig}
\end{figure}

Figure \ref{KF_fig2} offers another way to view this phenomenon. For $K_d < 1.65 \cdot 10^{-4}$, there are two SIPH points lying to the left of $A\approx 0.06$. As $K_d$ increases toward $1.65 \cdot 10^{-4}$, these two SIPH points approach one another and eventually coalesce, producing a broad plateau-like region in the graph of the period input-output function. For $K_d > 1.65 \cdot 10^{-4}$, the period input-output function is strictly decreasing in this region; hence, there is no longer IPH near this value of $A$. This geometry characterizes the chair singularity as an organizing center for the period input-output function: as an additional parameter is varied, the function changes from having two nearby SIPH points to being strictly monotone. Consequently, chair singularities provide a mechanism by which a broad plateau-like region can appear or disappear under variation of a parameter. Although this behavior lies beyond the theory developed in this paper, Figure \ref{KF_fig2} motivates the study of chair infinitesimal period homeostasis (CIPH) near Hopf bifurcation.
        
\begin{figure}[h!] 
                \begin{subfigure}{0.48\textwidth}
                \includegraphics[width=\linewidth]{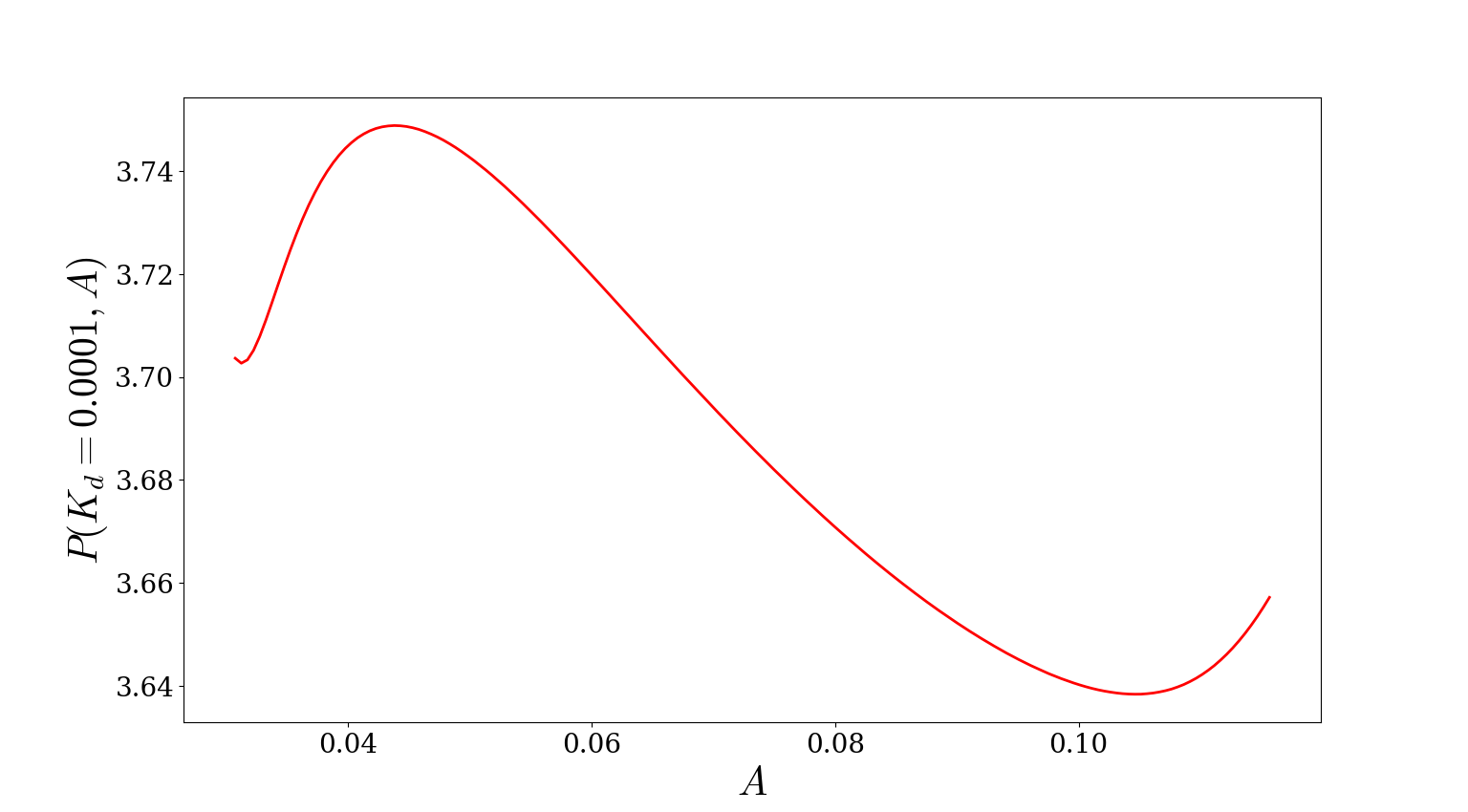}
                \caption{} \label{KF_fig2_subA}
                \end{subfigure}\hspace*{\fill}
                \begin{subfigure}{0.48\textwidth}
                \includegraphics[width=\linewidth]{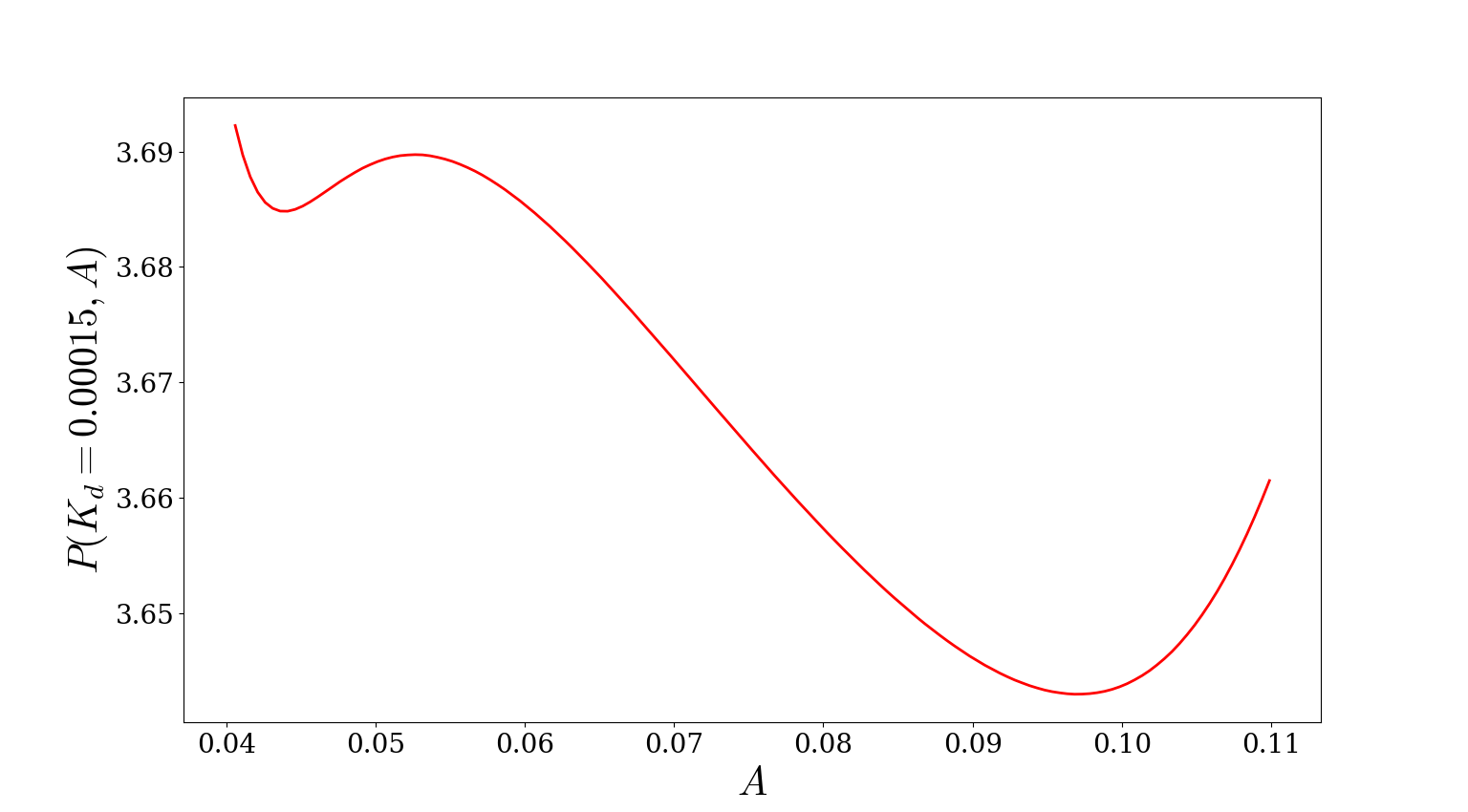}
                \caption{} \label{KF_fig2_subB}
                \end{subfigure}
                
                \medskip
                \begin{subfigure}{0.48\textwidth}
                \includegraphics[width=\linewidth]{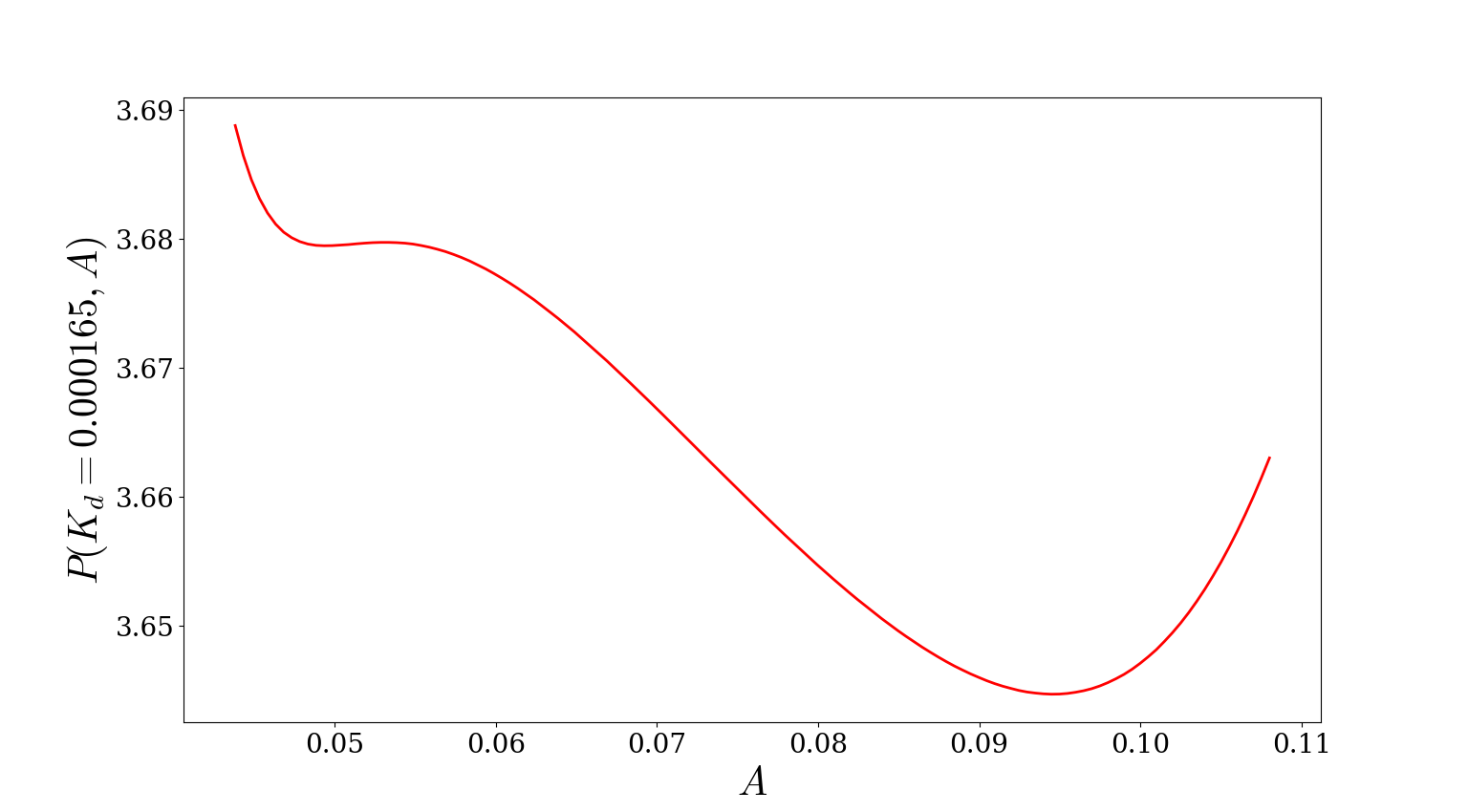}
                \caption{} \label{KF_fig2_subC}
                \end{subfigure}\hspace*{\fill}
                \begin{subfigure}{0.48\textwidth}
                \includegraphics[width=\linewidth]{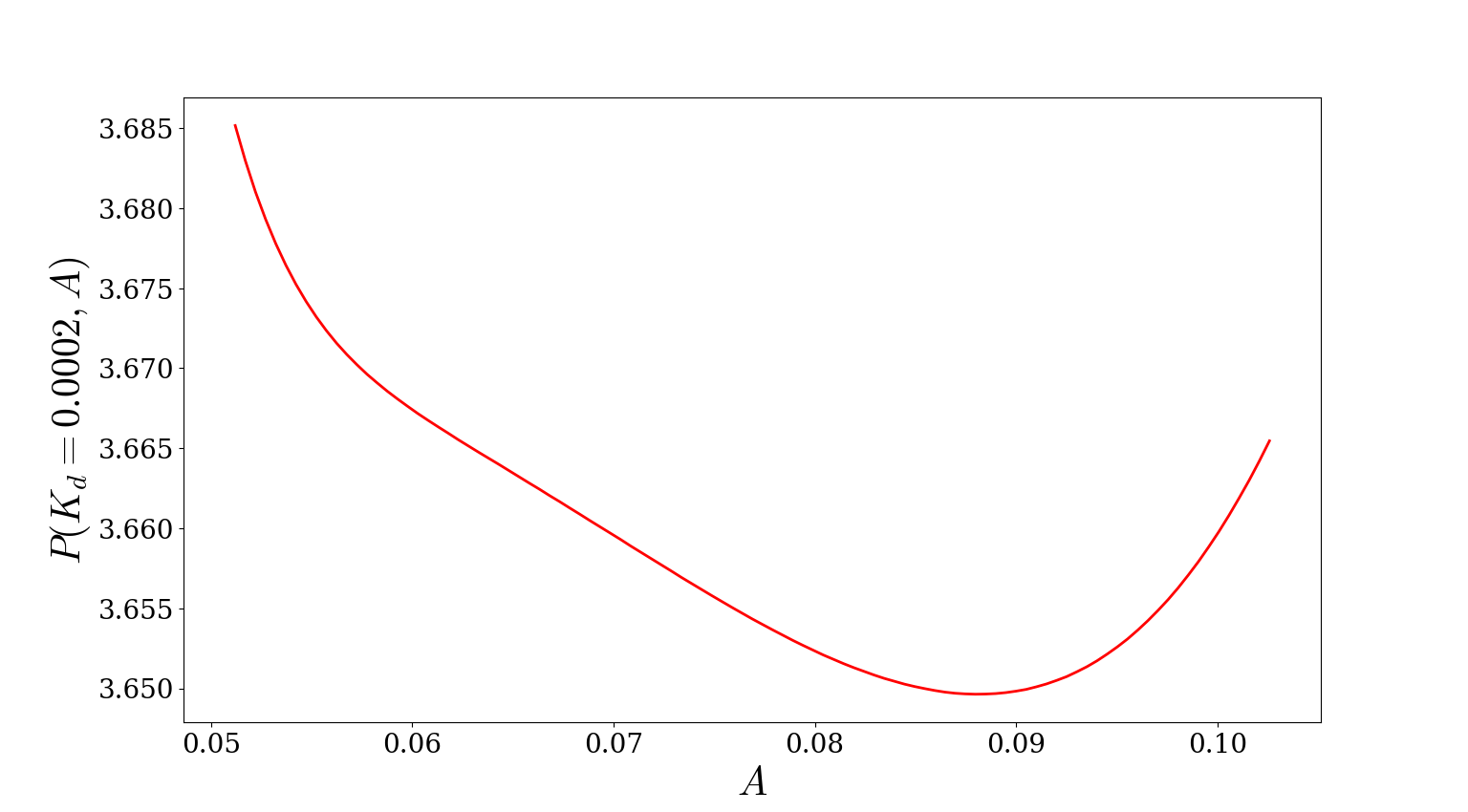}
                \caption{} \label{KF_fig2_subD}
                \end{subfigure}
                \caption{\small Plots of the period input-output function as a function of the input parameter $A$ as the bifurcation parameter $K_d$ is increased from $1\cdot 10^{-4}$ to $2\cdot 10^{-4}$. In Figures \ref{KF_fig2_subA} and \ref{KF_fig2_subB}, there are two SIPH points to the left of $A\approx 0.06$. In Figure \ref{KF_fig2_subC}, these two SIPH points have coalesced to produce a higher-order form of IPH that is suggestive of a chair-type singularity \cite{Golubitsky2017}. Figure \ref{KF_fig2_subD} shows that after further increasing the value of $K_d$, there is no longer IPH near this value of $A$.} \label{KF_fig2}
\end{figure}

\section{Results} \label{results_sec}

Our main result, Theorem \ref{simple_thm}, was described in Section \ref{intro_simultaneous_hopf_iph}. We now present a proof of this result and its consequences. Section \ref{liapunov_schmidt_sec} begins by using Liapunov--Schmidt reduction to derive a smooth period input-output function. This derivation also yields a formula for the period, and hence a way to find IPH points. Section \ref{siph_sec} then proves Theorem \ref{simple_thm} and derives consequences that qualitatively characterize solutions of \eqref{homeo_defn_eqn}. Theorem \ref{classification_thm} follows immediately from these consequences. Finally, Section \ref{taylor_coeff_sec} computes the relevant Taylor coefficients of $T$ in terms of those of $p$ and $q$ from the Liapunov--Schmidt normal form. These calculations help determine whether a given model satisfies the SIPH hypotheses (Definition \ref{siph_assumption_defn}).

\subsection{Derivation of the Period Input-Output Function}\label{liapunov_schmidt_sec}

Recall from Section \ref{intro_hopf} the two results we are using from bifurcation theory: reducing to an $S^1$-equivariant map on $\R^2$ via the Liapunov--Schmidt procedure and the Hopf bifurcation theorem. A proof of the former can be found in Chapter \RN{8} of \cite{Golubitsky1985}. Although the Hopf bifurcation theorem is a classic result, the formula for the period is often passed over, but it is essential for our purposes. The goal of this section is therefore to prove Theorem \ref{hopf_thm}, with particular emphasis on deriving \eqref{period_eqn}.

\begin{proof}
Since the simple eigenvalue condition \ref{H1} is satisfied, we can use Theorem \ref{reduced_eqn_thm} to reduce to Liapunov--Schmidt normal form. After performing this reduction, nontrivial small-amplitude periodic solutions of \eqref{homeo_defn_eqn} with period $\frac{2\pi}{1 + \tau}$ correspond to simultaneous zeros of $p$ and $q$ with $x^2 + y^2 > 0$. Furthermore, the functions $p$ and $q$ are invariant under rotations, so we can assume without loss of generality that $y=0$ and $x\geq0$. Writing $z=x^2$, the system that must be solved is
\begin{align*}
p(z, \lambda, \cI, \tau) = 0\qquad \text{and}\qquad q(z, \lambda, \cI, \tau) = 0
\end{align*}
We now show the Hopf hypotheses (Definition \ref{hopf_assumption_defn}) imply that solutions of this system near the origin are uniquely determined by $(\lambda, \cI)$. Starting with the $q$ equation,  recall from \eqref{p_q_eqns} that $q(0)=0$ and $q_\tau(0)\neq 0$. Therefore, it follows from the implicit function theorem that there exists a unique smooth function $\tilde{T}(z, \lambda, \cI)$, defined on a neighborhood of $(z, \lambda, \cI) = (0,0,0)$, such that
\begin{align*}
\tilde{T}(0) = 0\qquad \text{and}\qquad q(z, \lambda, \cI, \tilde{T}(z, \lambda, \cI))\equiv 0
\end{align*}
Now, observe that differentiating the $p$ equation with respect to $z$ and evaluating at $(z, \lambda, \cI) = (0,0, 0)$ yields
\begin{align*}
p_{z}(0) + p_{\tau}(0)\tilde{T}_{z}(0) = p_{z}(0)\neq 0
\end{align*}
The equality above holds due to \eqref{p_q_eqns}, while the inequality is precisely the cubic condition \ref{H3}. Furthermore, $p(0) = 0$ by \eqref{p_q_eqns}. Hence, another application of the implicit function theorem gives the existence of a unique smooth function $\tilde{Z}(\lambda, \cI)$, defined on a neighborhood of $(\lambda, \cI) = (0, 0)$, such that
\begin{align*}
\tilde{Z}(0) &= 0\qquad p(\tilde{Z}(\lambda, \cI), \lambda, \cI, \tilde{T}(\tilde{Z}(\lambda, \cI), \lambda, \cI))\equiv 0 \qquad q(\tilde{Z}(\lambda, \cI), \lambda, \cI, \tilde{T}(\tilde{Z}(\lambda, \cI), \lambda, \cI))\equiv 0
\end{align*}
It follows from this argument that the period parameter $\tau$ is uniquely determined by the smooth function
\begin{align*}
T(\lambda, \cI) = \tilde{T}(\tilde{Z}(\lambda, \cI),\lambda, \cI)
\end{align*}
which completes the proof of \eqref{period_eqn}. 

What remains to be accounted for is uniqueness and stability of the periodic orbit of \eqref{homeo_defn_eqn}. The proofs of these facts are not central to our analysis, so we refer the interested reader to Theorem 3.2 (p. 352) and Theorem 4.1 (p. 360) in \cite{Golubitsky1985}.
\end{proof}


\subsection{Theory of Simultaneous Hopf Bifurcation and Period Homeostasis}\label{siph_sec}

Section \ref{liapunov_schmidt_sec} showed that, under the Hopf hypotheses (Definition \ref{hopf_assumption_defn}), there exists a well-defined period input-output function. We now use this fact to prove Theorem \ref{simple_thm}.

\begin{proof}
We first show that there exists a unique smooth nondegenerate Hopf curve $\cH = \{(\lambda, \cI): \cI = \cI_{\hb}(\lambda)\}$ such that $\cI_{\hb}(0) = 0$. To that end, recall from the simple eigenvalue condition \ref{H1} that all eigenvalues of $(\mathrm{d}F)_{0, 0, 0}$ other than $\pm i$ have strictly negative real part. Since the eigenvalues of $(\mathrm{d}F)_{0, \lambda, \cI}$ depend smoothly on $(\lambda, \cI)$, the noncritical eigenvalues continue to have strictly negative real part on a neighborhood of the origin. By a similar argument, the Hopf crossing \ref{H2} and cubic \ref{H3} conditions hold on a neighborhood of the origin. Therefore, nondegenerate Hopf bifurcation occurs at solutions of the equation
\begin{align*}
\sigma(\lambda, \cI) = 0
\end{align*}
Since $\sigma(0)=0$ and $\sigma_{\cI}(0)\neq 0$, it follows from the implicit function theorem that, locally, this equation uniquely determines $\cI$ as a smooth function of $\lambda$. Let $\cI = \cI_{\hb}(\lambda)$ be such a function. Then
\begin{align*}
\cI_{\hb}(0) = 0 \qquad \text{and} \qquad \sigma(\lambda, \cI_{\hb}(\lambda)) \equiv 0
\end{align*}
which concludes the proof for the Hopf curve.

Next, we show that there exists a unique smooth IPH curve 
\begin{align*}
\cS = \{(\lambda, \cI): \text{ a periodic solution exists and } \cI = \cI_{\ph}(\lambda)\}
\end{align*}
 such that $\cI_{\ph}(0) = 0$ and $P_{\cI \cI}(\lambda, \cI_{\ph}(\lambda))\neq 0$. To see this is true, recall from the curve condition \ref{P2} that $T_{\cI}(0)=0$ and $T_{\cI \cI}(0)\neq 0$. Therefore, the implicit function theorem asserts the existence of a unique smooth function $\cI = \cI_{\ph}(\lambda)$, defined on a neighborhood of $\lambda = 0$, such that
\begin{align*}
\cI_{\ph}(0) = 0 \qquad \text{and}\qquad T_{\cI}(\lambda, \cI_{\ph}(\lambda))\equiv 0
\end{align*}
Furthermore, $T_{\cI \cI}(\lambda, \cI_{\ph}(\lambda))\neq 0$ holds by continuity. To complete the proof for the IPH curve, observe that \eqref{period_eqn} implies
\begin{align*}
     P_{\cI}(\lambda_0, \cI_0) = 0 \qquad \text{and} \qquad P_{\cI \cI}(\lambda_0, \cI_0) \neq 0
\end{align*}
if and only if
\begin{align*}
     T_{\cI}(\lambda_0,\cI_0) = 0 \qquad \text{and}\qquad T_{\cI \cI}(\lambda_0, \cI_0) \neq 0
\end{align*}
In other words, the defining and nondegeneracy conditions for SIPH can equivalently be stated in terms of either $T$ or the period input-output function. Hence, the function $\cI = \cI_{\ph}(\lambda)$ satisfies all of the desired properties, so this part of the proof is complete. 

Finally, we prove that the intersection of $\cH$ and $\cS$ at the origin is transverse. In light of the description of $\cH$ and $\cS$ as the graphs of $\cI = \cI_{\hb}(\lambda)$ and $\cI = \cI_{\ph}(\lambda)$, what we must show is that
\begin{align*}
\cI_{\hb}'(0)\neq \cI_{\ph}'(0)
\end{align*}
We verify this inequality by computing the left and right-hand sides, then recognizing it is equivalent to the transversality condition \ref{P3}. Starting with the left-hand side, differentiating the identity $\sigma \equiv 0$ with respect to $\lambda$ and evaluating at $\lambda = 0$ yields 
\begin{align*}
\cI_{\hb}'(0) = -\lp \frac{\sigma_{\lambda}}{\sigma_{\cI}}\rp (0)
\end{align*}
Similarly, it follows from the identity $T_{\cI} \equiv 0$ that
\begin{align*}
\cI_{\ph}'(0) = -\lp \frac{T_{\cI \lambda}}{T_{\cI \cI}}\rp(0)
\end{align*}
Therefore, $\cI_{\hb}'(0)\neq \cI_{\ph}'(0)$ if and only if $(\sigma_{\lambda}T_{\cI \cI} - \sigma_{\cI}T_{\cI \lambda})(0)\neq 0$, which completes the proof of Theorem \ref{simple_thm}.
\end{proof}

Theorem \ref{simple_thm} guarantees the existence of an HH graph; however, in Section \ref{intro_simultaneous_hopf_iph} we pointed out that the significance of an HH graph is that it characterizes the qualitative behavior of solutions of \eqref{homeo_defn_eqn}. The rest of this section is devoted to deriving consequences of Theorem \ref{simple_thm} that demonstrate this is true. 

We begin by showing that, as in Figure \ref{siph_example_subfigA}, the Hopf curve $\cH$ divides a neighborhood of the origin in the $\lambda \cI$-plane into steady-state and periodic regimes. To set up some notation used throughout our analysis, recall from Section \ref{intro_simultaneous_hopf_iph} that
\begin{align}
\cH^c &= \{(\lambda, \cI): (\lambda, \cI) \not \in \cH\}\nonumber\\
&= \{(\lambda, \cI): \cI < \cI_{\hb}(\lambda)\} \dcup \{(\lambda,\cI): \cI > \cI_{\hb}(\lambda)\}\nonumber\\
&=\cH_{-} \dcup \cH_{+} \label{u_i_tilde_defn}
\end{align}
Furthermore, since the function $\eta(\lambda, \cI) = \cI - \cI_{\hb}(\lambda)$ is continuous, it follows from the fact that $\cH_{-} = \eta^{-1}(-\infty, 0)$ and $\cH_{+} = \eta^{-1}(0, \infty)$ that $\cH_{-}$ and $\cH_{+}$ are open sets. 

The first consequence we derive from Theorem \ref{simple_thm} is that the sets defined by \eqref{u_i_tilde_defn} do indeed correspond to steady-state and periodic regimes. As we now show, the role played by each set is determined by the sign of $\sigma_{\cI}(0)$. 

\begin{corollary}\label{sigma_i_cor}
Suppose that \eqref{homeo_defn_eqn} satisfies the Hopf (Definition \ref{hopf_assumption_defn}) and SIPH (Definition \ref{siph_assumption_defn}) hypotheses. Then, if $\sigma_{\cI}(0) > 0$, $\cH_{-}$ and $\cH_{+}$ correspond to the steady-state and periodic regimes, respectively. Conversely, if $\sigma_{\cI}(0) <0$, then the roles of $\cH_{-}$ and $\cH_{+}$ are reversed.
\end{corollary}

\begin{proof}
Suppose $\sigma_{\cI}(0)>0$. We will show that when $(\lambda, \cI) \in \cH_{-}$ the only solution of \eqref{homeo_defn_eqn} is a stable steady state, while there is an unstable steady state and stable periodic orbit when $(\lambda, \cI)\in \cH_{+}$. The argument for the case where $\sigma_{\cI}(0)<0$ is analogous, so this will complete the proof.

Let $(\lambda_1, \cI_1)\in \cH_{-}$. By definition of $\cH_{-}$, $\cI_1 < \cI_{\hb}(\lambda_1)$. Since $\sigma$ is an increasing function of $\cI$ near the origin, it follows that
\begin{align*}
\sigma(\lambda_1, \cI_1) < \sigma(\lambda_1, \cI_{\hb}(\lambda_1)) = 0
\end{align*}
Hence, the real part of the eigenvalues $\sigma \pm i \omega$ is strictly negative at $(\lambda_1, \cI_1)$. All other eigenvalues of $(\mathrm{d}F)_{0, \lambda_1, \cI_1}$ also have strictly negative real part, so it follows from linear stability analysis that the steady-state solution $u\equiv 0$ is stable.

Now, let $(\lambda_2, \cI_2)\in \cH_{+}$. Since there is an exchange of stability in nondegenerate Hopf bifurcation, it suffices to show that the steady-state solution $u\equiv 0$ is unstable. To see this is true, recall that $(\lambda_2, \cI_2) \in \cH_{+}$ if and only if $\cI_2 > \cI_{\hb}(\lambda_2)$. Using the fact that $\sigma$ is increasing with respect to $\cI$ near the origin, this implies
\begin{align*}
\sigma(\lambda_2, \cI_2) > \sigma(\lambda_2, \cI_{\hb}(\lambda_2))= 0
\end{align*}
Therefore, it follows from linear stability analysis that $u\equiv 0$ is unstable at $(\lambda_2, \cI_2)$, which completes the proof. 
\end{proof}

Recall that in Figure \ref{siph_example_subfigA} the IPH curve $\cS$ further divides the periodic regime according to the monotonicity of the period input-output function with respect to $\cI$. Given the sign of $\sigma_{\cI}(0)$, we now know from Corollary \ref{sigma_i_cor} which of the sets $\cH_{-}$, $\cH_{+}$ possesses a periodic orbit. Therefore, we can prove that the partition into three regions shown in Figure \ref{siph_example_subfigA} holds for general systems.

\begin{corollary}\label{hh_partition_cor}
Suppose that \eqref{homeo_defn_eqn} satisfies the Hopf (Definition \ref{hopf_assumption_defn}) and SIPH (Definition \ref{siph_assumption_defn}) hypotheses. Then the HH graph partitions a neighborhood of the origin in the $\lambda \cI$-plane into three open sets.
\end{corollary} 

\begin{proof}
According to \eqref{u_i_tilde_defn}, the Hopf curve $\cH$ splits the neighborhood into two open sets. What remains to be shown is that the IPH curve $\cS$ splits precisely one of $\cH_{-}$, $\cH_{+}$ into two open sets. While showing this is true, we use the fact that a transverse intersection of two curves in a plane is isolated \cite{Lee2013}, so $\cI_{\hb}(\lambda) \neq \cI_{\ph}(\lambda)$ for all $\lambda \neq 0$ in a neighborhood of the origin. 

Let $\Gamma = \{(\lambda, \cI): \cI = \cI_{\ph}(\lambda)\}$ be the graph of the function $\cI_{\ph}(\lambda)$. Although $\Gamma$ and $\cS$ are closely related, they are not the same. Rather, $\cS$ is the part of $\Gamma$ that lies in the periodic regime. With this point in mind, we first consider how $\cH$ and $\Gamma$ partition a neighborhood of the origin in the $\lambda\cI$-plane.
\begin{align}
\Gamma^c \cap \cH_{-} &= \{(\lambda, \cI)\in \cH_{-}: \cI < \cI_{\ph}(\lambda)\} \dcup \{(\lambda, \cI)\in \cH_{-}: \cI > \cI_{\ph}(\lambda)\}\nonumber\\
&= U_1\dcup U_2 \label{u_12_defn} \\
\Gamma^c\cap \cH_{+} &= \{(\lambda, \cI)\in \cH_{+}: \cI < \cI_{\ph}(\lambda)\}\dcup \{(\lambda, \cI)\in \cH_{+}:\cI > \cI_{\ph}(\lambda)\}\nonumber\\
&=U_3\dcup U_4 \label{u_34_defn}
\end{align}
Combining \eqref{u_i_tilde_defn}-\eqref{u_34_defn} yields
\begin{align*}
(\cH \cup \Gamma)^c &= \cH^c \cap \Gamma^c\\
&=(\cH_{-}\dcup \cH_{+}) \cap \Gamma^c\\
&=(\cH_{-}\cap \Gamma^c) \dcup (\cH_{+}\cap \Gamma^c)\\
&= U_1\dcup U_2\dcup U_3 \dcup U_4
\end{align*}
Thus, $\cH$ and $\Gamma$ partition a neighborhood of the origin into four sets. Furthermore, each $U_i$ is an open set, since
\begin{align*}
U_1 = \cH_{-}\cap \{(\lambda, \cI): \cI < \cI_{\ph}(\lambda)\}
\end{align*}
That is, $U_1$ is the intersection of two open sets, as are $U_2$, $U_3$, and $U_4$.

We conclude the proof by using Corollary \ref{sigma_i_cor} to identify which of the sets $\cH_{-}$, $\cH_{+}$ contain a periodic orbit, then use observations made above to determine the partition. If $\sigma_{\cI}(0) > 0$, then $\cH_{+}$ is the periodic regime. Hence,
\begin{align*}
(\cH\cup \cS)^c &= \cH_{-}\dcup U_3 \dcup U_4
\end{align*}
Conversely, if $\sigma_{\cI}(0)<0$, then $\cH_{-}$ is the periodic regime. In this case, 
\begin{align*}
(\cH \cup \cS)^c &= U_1 \dcup U_2 \dcup \cH_{+}
\end{align*}
In either case, the HH graph partitions a neighborhood of the origin into three open sets, as we wished to show. 
\end{proof}

Corollary \ref{hh_partition_cor} shows that our representative HH graph, Figure \ref{siph_example_subfigA} in Section \ref{intro_simultaneous_hopf_iph}, contains the correct number of regions. However, we have yet to formalize what distinguishes the two regions within the periodic regime. Our next result shows that the key to this distinction is the sign of $T_{\cI \cI}(0)$.

\begin{corollary}\label{monotonicity_cor}
Suppose that \eqref{homeo_defn_eqn} satisfies the Hopf (Definition \ref{hopf_assumption_defn}) and SIPH (Definition \ref{siph_assumption_defn}) hypotheses, and let
\begin{align*}
(V_1, V_2) &=
\begin{cases}
(U_3, U_4) \text{ if }\sigma_{\cI}(0)>0\\
(U_1, U_2) \text{ if } \sigma_{\cI}(0) < 0
\end{cases}
\end{align*}
If $T_{\cI \cI}(0)>0$, then the period input-output function is increasing in $V_1$ and decreasing in $V_2$. Conversely, if $T_{\cI \cI}(0)<0$, then the period input-output function is increasing in $V_2$ and decreasing in $V_1$.
\end{corollary}

\begin{proof}
The proof is analogous for any combination of signs, so we assume without loss of generality that $\sigma_{\cI}(0) > 0$ and $T_{\cI \cI}(0)>0$. It must be shown that the period input-output function is increasing in $V_1 = U_3$ and decreasing in $V_2 = U_4$.

Let $(\lambda_1, \cI_1)\in V_1 = U_3$. It follows from the definition of $U_3$ \eqref{u_34_defn} that $\cI_1 < \cI_{\ph}(\lambda_1)$. Therefore, since $T_{\cI}(\lambda, \cI)$ is increasing with respect to $\cI$ near the origin, 
\begin{align*}
T_{\cI}(\lambda_1, \cI_1) < T_{\cI}(\lambda_1, \cI_{\ph}(\lambda_1))=0
\end{align*}
This shows that $T(\lambda, \cI)$ is decreasing with respect to $\cI$ in $V_1 = U_3$. Recall from \eqref{period_eqn} that $T(\lambda,\cI)$ is inversely related to the period input-output function, so it follows that the period input-output function is increasing with respect to $\cI$ in $V_1 = U_3$.

Now, let $(\lambda_2, \cI_2) \in V_2 = U_4$. By the definition of $U_4$ \eqref{u_34_defn}, $\cI_2 > \cI_{\ph}(\lambda_2)$, so 
\begin{align*}
T_{\cI}(\lambda_2, \cI_2) > T_{\cI}(\lambda_2, \cI_{\ph}(\lambda_2))= 0
\end{align*}
Hence, $T(\lambda, \cI)$ is increasing with respect to $\cI$ in $V_2 = U_4$, which implies that the period input-output function is decreasing with respect to $\cI$ in $V_2 = U_4$. This concludes the proof. 
\end{proof}

The final consequence we derive from Theorem \ref{simple_thm} identifies the normal form of the period input-output function near the IPH curve $\cS$. More precisely, after a smooth local change of coordinates, the period input-output function depends parabolically on $\cI$, and its variation is smallest near $\cS$.

\begin{corollary}\label{normal_form_cor}
Suppose that \eqref{homeo_defn_eqn} satisfies the Hopf (Definition \ref{hopf_assumption_defn}) and SIPH (Definition \ref{siph_assumption_defn}) hypotheses. For any $(\lambda_0, \cI_0)\in \cS$, there exists a smooth local change of coordinates such that $P(\lambda_0, \cI) = - \sgn(T_{\cI \cI}(0))\cI^2$.
\end{corollary}

\begin{proof}
We prove the corollary using a result from elementary catastrophe theory. Let $(\lambda_0, \cI_0)\in \cS$. It follows from Theorem \ref{simple_thm} that
\begin{align*}
P_{\cI}(\lambda_0, \cI_0) = 0 \qquad \text{and}\qquad P_{\cI \cI}(\lambda_0, \cI_0)\neq 0
\end{align*}
Therefore, after making the change of coordinates $\cI \mapsto \cI + \cI_0$, $P(\lambda_0, \cI)$ satisfies the hypotheses of Theorem 4.4 in \cite[pp. 57-58]{Poston1978}. It follows that there exists a smooth local change of coordinates such that
\begin{align*}
P(\lambda_0, \cI) = \pm \cI^2
\end{align*}
What remains to be shown is that the appropriate choice of sign in the normal form is determined by the sign of $T_{\cI \cI}(0)$. To see this is true, observe that it follows from \eqref{period_eqn} that
\begin{align*}
\sgn P_{\cI \cI}(\lambda_0, 0) = - \sgn T_{\cI \cI}(\lambda_0, \cI_0)
\end{align*}
Suppose $T_{\cI \cI}(0) > 0$. Then the equality above implies $P(\lambda_0, \cI)$ is concave at $\cI = 0$. Hence, the normal form is
\begin{align*}
P(\lambda_0, \cI) = -\cI^2 =  -\sgn (T_{\cI \cI}(0))\cI^2
\end{align*}
On the other hand, if $T_{\cI \cI}(0) < 0$, then $P(\lambda_0, \cI)$ is convex at $\cI = 0$. It follows that the normal form is
\begin{align*}
P(\lambda_0, \cI) = \cI^2 = - (-1)\cI^2 = -\sgn(T_{\cI \cI}(0))\cI^2
\end{align*}
Therefore, in either case the normal form is $-\sgn(T_{\cI \cI}(0))\cI^2$, as we wished to show.
\end{proof}


\subsection{Calculation of Taylor Coefficients of $T$}\label{taylor_coeff_sec}
Having proved our main result, we now consider how to verify whether a given model satisfies the Hopf (Definition \ref{hopf_assumption_defn}) and SIPH (Definition \ref{siph_assumption_defn}) hypotheses. Since nondegenerate Hopf bifurcation is classical, we focus on the SIPH hypotheses. These depend on the quantities
\begin{align*}
\sigma_{\lambda}(0) \qquad \sigma_{\cI}(0) \qquad T_{\cI}(0) \qquad T_{\cI \cI}(0) \qquad T_{\cI \lambda}(0)
\end{align*}
Proposition 3.3 in \cite[p. 352]{Golubitsky1985} handles the calculation of $\sigma_{\lambda}(0)$ and $\sigma_{\cI}(0)$. Therefore, it remains to determine $T_{\cI}(0)$, $T_{\cI \cI}(0)$, and $T_{\cI \lambda}(0)$.

We proceed as follows. The Hopf bifurcation literature contains formulas expressing Taylor coefficients of $p$ and $q$ in the Liapunov--Schmidt normal form in terms of the vector field $F$ in \eqref{homeo_defn_eqn}. For example, see Farr et al. \cite{Farr1989} and Golubitsky and Langford \cite{Golubitsky1981}. Thus, once the relevant Taylor coefficients of $T$ have been expressed in terms of those of $p$ and $q$, one can either appeal directly to these references or adapt their methods. To keep the discussion within the scope of this paper, we restrict attention to expressing $T_{\cI}(0)$, $T_{\cI \cI}(0)$, and $T_{\cI \lambda}(0)$ in terms of Taylor coefficients of $p$ and $q$. 

This can be done using implicit differentiation. Recall the identities
\begin{align*}
p(\tilde{Z}(\lambda, \cI), \lambda,\cI, T(\lambda,\cI))\equiv 0 \qquad \text{and}\qquad q(\tilde{Z}(\lambda, \cI), \lambda, \cI, T(\lambda, \cI))\equiv 0
\end{align*}
Although our goal is to compute $T_{\cI}(0)$, $T_{\cI \cI}(0)$ and $T_{\cI \lambda}(0)$, it follows from the identities above that these quantities depend on Taylor coefficients of $\tilde{Z}$. We therefore proceed in two steps. First, we differentiate $p\equiv 0$ to solve for the Taylor coefficients of $\tilde{Z}$. This is possible because $p_{\tau}(0) = 0$. Second, we differentiate $q\equiv 0$, substitute the formulas obtained in the first step, and solve for the Taylor coefficients of $T$. Since we need second-order coefficients, we carry out this procedure twice. 

To begin, we calculate linear terms in the Taylor series of $\tilde{Z}$.

\begin{proposition}\label{linear_z_prop}
Suppose that \eqref{homeo_defn_eqn} satisfies the Hopf (Definition \ref{hopf_assumption_defn}) and SIPH (Definition \ref{siph_assumption_defn}) hypotheses. Then
\begin{align*}
\tilde{Z}_{\cI}(0) = - \frac{p_{\cI}}{p_{z}}(0) \qquad \text{and} \qquad \tilde{Z}_{\lambda}(0) = - \frac{p_{\lambda}}{p_{z}}(0)
\end{align*}
\end{proposition}

\begin{proof}
The two calculations are analogous, so it suffices to compute $\tilde{Z}_{\cI}(0)$. Differentiate $p\equiv 0$ with respect to $\cI$ and evaluate at $(\lambda, \cI) = (0,0)$ to obtain
\begin{align*}
(p_{z}\tilde{Z}_{\cI} + p_{\cI} + p_{\tau}T_{\cI})(0) = 0
\end{align*}
Since $p_{\tau}(0) = 0$, solving for $\tilde{Z}_{\cI}(0)$ yields the formula stated in the proposition.
\end{proof}

Using Proposition \ref{linear_z_prop}, we now compute the linear terms in the Taylor series of $T$.

\begin{proposition}\label{linear_t_prop}
Suppose that \eqref{homeo_defn_eqn} satisfies the Hopf (Definition \ref{hopf_assumption_defn}) and SIPH (Definition \ref{siph_assumption_defn}) hypotheses. Then
\begin{align*}
T_{\cI}(0) = \frac{p_{z}q_{\cI} - q_{z}p_{\cI}}{p_{z}}(0) \qquad \text{and} \qquad T_{\lambda}(0) = \frac{p_{z}q_{\lambda}-q_{z}p_{\lambda}}{p_{z}}(0)
\end{align*}
\end{proposition}

\begin{proof}
We compute $T_{\cI}(0)$. By differentiating $q\equiv 0$ with respect to $\cI$ and evaluating at $(\lambda, \cI) = (0,0)$, we obtain
\begin{align*}
(q_{z}\tilde{Z}_{\cI} + q_{\cI}  + q_{\tau}T_{\cI})(0) = 0
\end{align*}
Recall from Theorem \ref{reduced_eqn_thm} that $q_{\tau}(0) = -1$, so it follows that
\begin{align*}
T_{\cI}(0) = (q_{z}\tilde{Z}_{\cI} + q_{\cI})(0) = \frac{p_{z} q_{\cI} - q_{z} p_{\cI}}{p_{z}}(0)
\end{align*}
This gives the formula for $T_{\cI}(0)$. The calculation of $T_{\lambda}(0)$ is analogous.
\end{proof}

We now turn to the quadratic terms in the Taylor series of $\tilde{Z}$. Since our goal is to compute $T_{\cI \cI}(0)$ and $T_{\cI \lambda}(0)$, we only need two of these coefficients. 

\begin{proposition}\label{quadratic_z_prop}
Suppose that \eqref{homeo_defn_eqn} satisfies the Hopf (Definition \ref{hopf_assumption_defn}) and SIPH (Definition \ref{siph_assumption_defn}) hypotheses. Then
\begin{align*}
\tilde{Z}_{\cI \cI}(0) &= \frac{2p_{z} p_{z\cI} p_{\cI} - p_{z z} p_{\cI}^2 - p_{z}^2 p_{\cI \cI}}{p_{z}^3}(0)\\
\tilde{Z}_{\cI \lambda}(0) &= \frac{1}{p_{z}^3}\lp p_{z} p_{z\lambda} p_{\cI} + p_{z \tau} p_{\cI} p_{z} q_{\lambda} + p_{z} p_{z\cI} p_{\lambda} + p_{z} p_{\tau \cI} p_{\lambda} q_{z} - p_{zz} p_{\lambda} p_{\cI} - p_{z\tau} p_{\cI} p_{\lambda} q_{z}\\
&\hspace{80mm} - p_{z}^2 p_{\cI \lambda} - p_{\tau \cI} p_{z}^2 q_{\lambda}\rp(0)
\end{align*}
\end{proposition}

\begin{proof}
We begin by computing $\tilde{Z}_{\cI \cI}(0)$. Differentiating $p\equiv 0$ with respect to $\cI$ yields
\begin{align}
0 &\equiv p_{z}(\tilde{Z}(\lambda, \cI), \lambda, \cI, T(\lambda,\cI)) \tilde{Z}_{\cI}(\lambda,\cI) + p_{\cI}(\tilde{Z}(\lambda, \cI), \lambda, \cI, T(\lambda,\cI))\nonumber\\
&\hspace{70mm}+ p_{\tau}(\tilde{Z}(\lambda,\cI), \lambda, \cI, T(\lambda, \cI)) T_{\cI}(\lambda, \cI)
\label{p_i_eqn}
\end{align}
Next, differentiate \eqref{p_i_eqn} with respect to $\cI$ and evaluate at $(\lambda, \cI) = (0,0)$ to obtain
\begin{align*}
0 = (p_{zz} \tilde{Z}_{\cI}^2 + 2 p_{z\cI} \tilde{Z}_{\cI} + p_{z} \tilde{Z}_{\cI \cI} + p_{\cI \cI})(0)
\end{align*}
After using Proposition \ref{linear_z_prop} to express $\tilde{Z}_{\cI}(0)$ in terms of Taylor coefficients of $p$ and $q$, solving this equation for $\tilde{Z}_{\cI \cI}(0)$ leads to the asserted formula.

We are left to compute $\tilde{Z}_{\cI \lambda}(0)$. The previous calculation simplified because $T_{\cI}(0) = 0$. Since $T_{\lambda}(0)$ is generally nonzero, the mixed derivative calculation involves more terms. To that end, differentiating \eqref{p_i_eqn} with respect to $\lambda$ and evaluating at $(\lambda,\cI) = (0,0)$ yields
\begin{align*}
0 &= (p_{zz} \tilde{Z}_{\lambda} \tilde{Z}_{\cI} + p_{z\lambda} \tilde{Z}_{\cI} + p_{z\tau} T_{\lambda} \tilde{Z}_{\cI} + p_{z} \tilde{Z}_{\cI \lambda} + p_{\cI z}\tilde{Z}_{\lambda} + p_{\cI \lambda} + p_{\cI \tau} T_{\lambda})(0)
\end{align*}
Substituting the formulas from Propositions \ref{linear_z_prop} and \ref{linear_t_prop} and solving for $\tilde{Z}_{\cI \lambda}(0)$ gives the desired result. 
\end{proof}

We conclude this section by expressing $T_{\cI \cI}(0)$ and $T_{\cI \lambda}(0)$ in terms of Taylor coefficients of $p$ and $q$. 

\begin{proposition}\label{quadratic_t_prop}
Suppose \eqref{homeo_defn_eqn} satisfies the Hopf (Definition \ref{hopf_assumption_defn}) and SIPH (Definition \ref{siph_assumption_defn}) hypotheses. Then
\begin{align*}
T_{\cI \cI}(0) &= \frac{1}{p_{z}^3} \lp 2 p_{z} q_{z} p_{z\cI} p_{\cI} + p_{z} q_{zz} p_{\cI}^2 + p_{z}^3 q_{\cI \cI} -q_{z} p_{zz} p_{\cI}^2 - p_{z}^2 q_{z} p_{\cI \cI} - 2 p_{z}^2 q_{z\cI} p_{\cI}\rp (0) \\
T_{\cI \lambda}(0) &= \frac{1}{p_{z}^3} \lp q_{z} p_{z} p_{z\lambda} p_{\cI} + q_{z} p_{z} p_{z\tau} q_{\lambda} p_{\cI} + q_{z} p_{z} p_{\cI z} p_{\lambda} + q_{z}^2 p_{z} p_{\cI \tau} p_{\lambda} + p_{z} q_{zz} p_{\lambda} p_{\cI} \\
&\hspace{10mm} + p_{z} q_{z\tau} p_{\cI} q_{z} p_{\lambda} + p_{z}^3 q_{\cI \lambda} + p_{z}^3 q_{\cI \tau} q_{\lambda} - q_{z} p_{zz} p_{\lambda} p_{\cI} - q_{z}^2 p_{z\tau} p_{\cI} p_{\lambda} - q_{z} p_{z}^2 p_{\cI \lambda}\\
&\hspace{10mm} - q_{z} p_{z}^2 p_{\cI \tau} q_{\lambda} - p_{z}^2 q_{z\lambda} p_{\cI} - p_{z}^2 q_{z\tau} p_{\cI} q_{\lambda} - p_{z}^2 q_{\cI z} p_{\lambda} - p_{z}^2 q_{\cI \tau} q_{z} p_{\lambda}\rp (0)
\end{align*}
\end{proposition}

\begin{proof}
Both calculations require us to differentiate $q\equiv 0$ with respect to $\cI$, which yields
\begin{align}
0 &\equiv q_{z}(\tilde{Z}(\lambda, \cI), \lambda,\cI, T(\lambda, \cI))\tilde{Z}_{\cI}(\lambda,\cI) + q_{\cI}(\tilde{Z}(\lambda,\cI), \lambda, \cI, T(\lambda,\cI)) \nonumber \\
&\hspace{70mm}+ q_{\tau}(\tilde{Z}(\lambda,\cI), \lambda, \cI, T(\lambda, \cI))T_{\cI}(\lambda, \cI)
\label{t_i_eqn}
\end{align}
Now, we first calculate $T_{\cI \cI}(0)$ by differentiating \eqref{t_i_eqn} with respect to $\cI$ and evaluating at $(\lambda, \cI) = (0,0)$. The result is
\begin{align*}
0 &= (q_{zz}\tilde{Z}_{\cI}^2 + 2 q_{z\cI} \tilde{Z}_{\cI} + q_{z} \tilde{Z}_{\cI \cI} + q_{\cI \cI} + q_{\tau} T_{\cI \cI})(0)
\end{align*}
Using $q_{\tau}(0) = -1$ along with the formulas for $\tilde{Z}_{\cI}(0)$ and $\tilde{Z}_{\cI \cI}(0)$ from Propositions \ref{linear_z_prop} and \ref{quadratic_z_prop}, respectively, and solving for $T_{\cI \cI}(0)$ gives the desired formula. 

To complete the proof, we compute $T_{\cI \lambda}(0)$. Since $T_{\lambda}(0)$ is generically nonzero, this calculation is more involved, but the method is the same. Namely, differentiate \eqref{t_i_eqn} with respect to $\lambda$ and evaluate at $(\lambda, \cI) = (0,0)$ to obtain
\begin{align*}
0 &= (q_{zz} \tilde{Z}_{\lambda} \tilde{Z}_{\cI} + q_{z\lambda} \tilde{Z}_{\cI} + q_{z\tau} T_{\lambda} \tilde{Z}_{\cI} + q_{z} \tilde{Z}_{\cI \lambda} + q_{\cI z} \tilde{Z}_{\lambda} + q_{\cI \lambda} + q_{\cI \tau} T_{\lambda} + q_{\tau}T_{\cI \lambda})(0)
\end{align*}
Substituting the previously derived formulas and solving for $T_{\cI \lambda}(0)$ gives the asserted formula.
\end{proof}

\section{Conclusion} \label{conclusion_sec}

In this paper, we studied how certain biological oscillators maintain a tightly regulated period despite variation in a parameter. Motivated by the notion of infinitesimal steady-state homeostasis introduced by Golubitsky and Stewart \cite{Golubitsky2017}, we formulated an infinitesimal notion of homeostasis for periodic solutions. Since periodic solutions often arise through Hopf bifurcation, we worked in that setting. Our main result shows that when nondegenerate Hopf bifurcation and SIPH occur simultaneously at a point in parameter space, the local behavior is organized by an HH graph. 

The HH graph provides a mathematical framework for understanding both features of interest: the existence of an oscillation and the robustness of the period. Namely, the oscillation is the result of a small-amplitude periodic orbit born at a Hopf bifurcation, while the robustness of the period is explained by IPH. Our simulation of the Kim--Forger SNF model identified the existence of three IPH curves, which supports these claims. 

However, there are still several important open questions. First, the theory assumes nondegeneracy on both the Hopf and IPH sides, whereas biological models of interest may exhibit degeneracies. On the Hopf side, degeneracy occurs when the simple eigenvalue condition \ref{H1} fails, the Hopf crossing condition \ref{H2} fails, or the cubic coefficient satisfies $p_z(0) = 0$. For example, Pei et al. \cite{Pei2024} show that the first Liapunov coefficient vanishes in the Kim--Forger SNF model at the point $\cG$ in Figure \ref{KF_fig1}, corresponding to $p_{z}(0) = 0$ in our formulation. On the IPH side, degeneracy occurs when one of the SIPH hypotheses (Definition \ref{siph_assumption_defn}) fails. Our simulations suggest that the Kim--Forger SNF model also provides an example of this, since Figure \ref{KF_fig2} exhibits the geometry associated with a chair singularity. Namely, two SIPH points approach, coalesce, and then disappear. Developing a theory for this transition would extend the current framework beyond nondegenerate IPH.

Second, certain properties of infinitesimal steady-state homeostasis are known to be closely connected to network architecture, where network architecture stems from biochemical reaction systems \cite{Wang2021, Duncan2024, Reed2017}. It is unclear whether a connection between biochemical network structure and period homeostasis exists. Investigating this possibility could help clarify aspects of period homeostasis.


\section*{Acknowledgements}
We thank Fernando Antoneli,  Adriana Dawes, Punit Gandhi, Jiaxin Jin, Liam O'Brien, and Yangyang Wang for insightful discussions and helpful feedback. 



\end{document}